\documentclass[a4paper,11pt]{article}
\usepackage{amsthm,amsmath,amsfonts,amssymb}
\usepackage{microtype}
\usepackage{graphicx,color}
\usepackage{calrsfs}
\usepackage{bm}
\usepackage[round,authoryear]{natbib}
\usepackage{mathtools}
\usepackage{mathrsfs}
\usepackage{graphicx}
\usepackage{bbm}
\usepackage{bbding}

\newcommand*{\xhat}[1]{#1\kern-0.65em\hat{\phantom{#1}}}
\newcommand*{\xtilde}[1]{#1\kern-0.75em\tilde{\phantom{#1}}}
\newcommand*{\xhatt}[1]{#1\kern-0.425em\hat{\phantom{#1}}}

\numberwithin{equation}{section}

\newtheorem{ass}{Assumption}[section]
\newtheorem{thm}[ass]{Theorem}
\newtheorem{pro}[ass]{Proposition}
\newtheorem{cor}[ass]{Corollary}
\newtheorem{lma}[ass]{Lemma}
\theoremstyle{definition}
\newtheorem{dfn}[ass]{Definition}

\begin{document}
\title{Modeling temporally uncorrelated components for complex-valued stationary processes}

\renewcommand{\thefootnote}{\fnsymbol{footnote}}

\author{Niko Lietz\'{e}n\footnotemark[1] \footnotemark[2] \, and \,  Lauri Viitasaari\footnotemark[3] \, and \, Pauliina Ilmonen\footnotemark[2]}

\footnotetext[1]{Correspondence to: Niko Lietz\'{e}n, E-mail: niko.lietzen@aalto.fi}
\footnotetext[2]{Aalto University School of Science, Department of Mathematics and Systems Analysis}
\footnotetext[3]{Aalto University School of Business, Department
of Information and Service Management}

\maketitle

\begin{abstract}
\noindent
We consider a complex-valued linear mixture model, under discrete weakly stationary processes. We recover latent components of interest, which have undergone a linear mixing. We study asymptotic properties of a classical unmixing estimator, that is based on simultaneous diagonalization of the covariance matrix and an autocovariance matrix with lag $\tau$. Our main contribution is that our asymptotic results can be applied to a large class of  processes. In related literature, the processes are typically assumed to have weak correlations. We extend this class and consider the unmixing estimator under stronger dependency structures. In particular, we analyze the asymptotic behavior of the unmixing estimator under both, long- and short-range dependent complex-valued processes. Consequently, our theory covers unmixing estimators that converge slower than the usual $\sqrt{T}$ and unmixing estimators that produce non-Gaussian asymptotic distributions. The presented methodology is a powerful prepossessing tool and highly applicable in several fields of statistics. Complex-valued processes are frequently encountered in, for example, biomedical applications and signal processing. In addition, our approach can be applied to model real-valued problems that involve temporally uncorrelated pairs. These are encountered in, for example, applications in finance.
\end{abstract}

{\small
\medskip

}

\section{Introduction}
In modern statistics, there is an increasing demand for theoretically solid methodology, which can be applied beyond standard real-valued data, into the realm of more exotic data structures. In this paper, we consider a complex-valued linear mixture model, based on temporally uncorrelated components, in the context of discrete weakly stationary processes. We aim to find latent processes of interest, when only linear mixtures of them are observable. The recovery of the latent processes of interest is referred to as the unmixing procedure. The properties of the recovered processes themselves can be the source of interest, or alternatively, the approach can be used to reduce multivariate models into several univariate models, which can simplify the modeling burden.

In the context of linear mixture models, the assumption of uncorrelated components, or stronger conditions that imply uncorrelated components, is considered to be natural in several applications, for example, in finance and signal processing, see the article \cite{fan2008modelling} and the books \cite{alexander2001market,comon2010}.
Applications, where the observed time series are naturally complex-valued are frequent in, e.g., signal processing and biomedical applications, see \cite{zarzoso2009robust,li2011application}. In such applications, the interest can lie in the shapes of the unobservable signals, such as, the shapes of three dimensional image valued signals, which correspond to different parts of the brain under different activities. In the signal processing literature, the problem is often referred to as the blind source separation (BSS) problem. In the BSS literature, it has been argued that discarding the special complex-valued structure can lead to lost information, see \cite{adali2014optimization}. Thus, it is often beneficial to directly work with complex-valued signals in such applications. We wish to emphasize that our $\mathbb{C}^d$-valued model does not directly correspond to existing $\mathbb{R}^{2d}$ valued models, e.g., the one in \cite{miettinen2012}.  For a collection of BSS applications, see \cite{comon2010}.

In parallel to the signal processing community, linear mixture models, with similar model assumptions as in BSS literature, have recently received notable attention in finance, see for example \cite{fan2008modelling, matteson2011dynamic}. In financial applications, the term blind source separation is rarely used, and usually more descriptive model naming conventions are utilized. Note that, our complex-valued approach is highly applicable in many real-valued financial applications. In our approach, we assume little concerning the relationship between the real- and imaginary-parts of a single component.  Thus, from the real-valued perspective, the problem can be equivalently considered as modeling real-valued temporally uncorrelated pairs, where the elements contained in a single pair are not necessarily uncorrelated. It turns out, that the algebra is often nicer in the complex-valued case, when compared to the notation required by the real-valued counterpart. Once the temporally uncorrelated components are recovered, one can then, for example, model volatilities bivariately (or univariately) or asses risk by modeling tail behavior with the tools of  bivariate (or univariate) extreme value theory. Moreover, our approach is natural in applications, where a single observation is vector valued, that is, the observations have both a magnitude and a direction, e.g., modeling the latent components of wind at an airport.

The main focus in this paper is on asymptotic behavior of a classical unmixing estimator.
We consider an algorithm that is identical to the so-called Algorithm for Multiple Unknown Signals Extraction (AMUSE), \cite{tong1990amuse}. In the financial side, asymptotic properties of an alternative approach, with slightly differing model assumptions compared to our approach,  are given in, e.g., \cite{fan2008modelling}. Asymptotics, in the context of real-valued BSS, have been considered in, e.g., \cite{miettinen2012,miettinen2016separation,miettinen2018extracting}. Compared to the above mentioned approaches,
we consider a substantially wider class of processes. In particular, we analyze the asymptotic behavior of the corresponding estimators under both, long- and short-range dependent complex-valued processes. Furthermore, we take a semi-parametric approach, that is, our theory is not limited to specific parametric family of distributions.
As a pinnacle of the novel theory presented in this paper, we consider processes without summable autocovariance structures, which causes  the limiting distribution of the corresponding unmixing estimator to be non-Gaussian. Instead of using the classical central limit theorem, we take a modern approach and utilize general central limit theorem type results that are also applicable for processes with strong temporal dependencies. Moreover,  we consider convergence rates that differ from the usual $\sqrt{T}$. We wish to emphasize that modeling long-range dependent processes, that is, processes without summable autocovariance structures, are of paramount importance in many financial applications.

The paper is structured as follows. In Section \ref{sec:ctimeseries}, we recall some basic theory regarding complex-valued random variables and we introduce the notation that we follow in this paper. In Section \ref{sec:bssmodel}, we present the temporally uncorrelated components mixing model and solutions to the corresponding unmixing problem at the population level. In Section \ref{sec:asympBSS}, we formulate the estimation procedure and study the asymptotic properties of the estimators. In Section \ref{sec:image}, we a real-data example involving photographs, which can be presented as complex-valued time series.   All the proofs of this paper are presented in Appendix \ref{app:B}.

\section{Random variables in the complex plane}
\label{sec:ctimeseries}
In this section, we review some basic definitions and classical estimators for complex-valued random variables. Additionally, we recall the definition of the complex multivariate Gaussian distribution and the classical complex-valued central limit theorem.

Throughout, let $(\Omega,\mathscr{F},\mathbb{P})$ be a common probability space and let $\textbf{z}_\textnormal{\tiny{\textbullet}} \coloneqq (\textbf{z}_t)_{t\in \mathbb{N}}$ be a collection of random variables $\textbf{z}_t : \Omega \rightarrow \mathbb{C}^d$, $t \in \mathbb{N}  = \{1,2,\ldots\}$, where the dimension parameter $d$ is a finite constant. Furthermore, let $z_t^{(k)} = \textnormal{pr}_k \circ \textbf{z}_t $, where $ \textnormal{pr}_k$ is a projection to the $k$th complex coordinate. We refer to the process ${z}^{(k)}_\textnormal{\tiny{\textbullet}} \coloneqq ({z}^{(k)}_t)_{t\in \mathbb{N}}$ as the $k$th component of the $\mathbb{C}^d$-valued stochastic process $\textbf{z}_\textnormal{\tiny{\textbullet}}$. Note that the components can be expressed in the form $z_t^{(k)} = a_t^{(k)} + ib_t^{(k)} $, where $a_t^{(k)}, b_t^{(k)}$ are real-valued $\forall k \in \left\{1,2,\ldots,d\right\}$ and $i$ is the imaginary unit. We denote the complex conjugate of $\textbf{z}_t$ as  $\textbf{z}_t^* $ and we denote the conjugate transpose of $\textbf{z}_t$ as $\textbf{z}_t^\textnormal{H} \coloneqq (\textbf{z}_t^*)^\top$. Furthermore, we use $\mathscr{B}(\mathbb{C}^d)$ and $\mathscr{B}(\mathbb{R}^{2d})$  to denote the Borel-sigma algebras on $\mathbb{C}^d$ and $\mathbb{R}^{2d}$, respectively.

We use the following notation for the multivariate mean and the unsymmetrized autocovariance matrix  with lag $\tau \in \left\{0\right\} \cup \mathbb{N} $,
\begin{align*}
&\boldsymbol{\mu}_{\textbf{z}_t} = \mathbb{E}\left[\textbf{z}_t\right]\in \mathbb{C}^d,
\quad
\ddot{\textbf{S}}_{\tau}\left[ \textbf{z}_t \right] = \mathbb{E}\left[\left(\textbf{z}_t - \boldsymbol{\mu}_{\textbf{z}_t}  \right)\left( \textbf{z}_{t+\tau} - \boldsymbol{\mu}_{\textbf{z}_{t+\tau}} \right)^\textnormal{H}\right]\in \mathbb{C}^{d\times d}.
\end{align*}
Furthermore, we use the following notation for the symmetrized autocovariance matrix with lag $\tau \in \left\{0\right\} \cup \mathbb{N} $,
\begin{align*}
{\textbf{S}}_{\tau}\left[ \textbf{z}_t \right] = \frac{1}{2} \left[ \ddot{\textbf{S}}_{\tau}\left[ \textbf{z}_t \right] +\left( \ddot{\textbf{S}}_{\tau}\left[ \textbf{z}_t \right]\right)^\textnormal{H} \right] \in \mathbb{C}^{d\times d}.
\end{align*}
In the case of univariate stochastic processes $x_\textnormal{\tiny{\textbullet}}$ and $y_\textnormal{\tiny{\textbullet}}$, we use $\ddot{\textnormal{S}}_{\tau}[x_t,y_s]$ and ${\textnormal{R}}_{\tau}[x_t,y_s]$   to denote,
\begin{align*}
&\ddot{\textnormal{S}}_{\tau}[x_t,y_s] = \mathbb{E}\left[ (x_t - \mu_{x_t})( y_{s+\tau} - \mu_{y_{s+\tau}})^* \right] \quad \textnormal{ and }\\
&\textnormal{R}_{\tau}[x_t,y_s] = \ddot{\textnormal{S}}_{\tau}[x_t,y_s] + \ddot{\textnormal{S}}_{\tau}[y_s,x_t].
\end{align*}
Note that the unsymmetrized autocovariance matrix $\ddot{\textbf{S}}_{\tau}$ is not necessarily conjugate symmetric, i.e., Hermite symmetric or Hermitian, and it can have complex-valued diagonal entries. In contrast, the symmetrized autocovariance matrix ${\textbf{S}}_{\tau}$ is always conjugate symmetric, that is, ${\textbf{S}}_{\tau} = {\textbf{S}}_{\tau}^\textnormal{H}$, and the diagonal entries of the symmetrized autocovariance matrix are always real-valued.

In  the literature, the definition of stationarity for complex-valued stochastic processes varies. In this paper, we use the following definitions for strict stationarity and weak stationarity.
\begin{dfn}
\label{def:strictstationarity}
The $\mathbb{C}^d$-valued process $\textbf{z}_\textnormal{\tiny{\textbullet}}  \coloneqq  \left(\textbf{z}_t\right)_{t\in \mathbb{N}} $ is strictly stationary if for any $\tau \in   \mathbb{N}$ and for any finite subset $S\subset \mathbb{N}$,
\begin{align*}
\left\{ \textbf{z}_{t+\tau}: t\in S\right\} \textnormal{ has the same distribution as } \left\{ \textbf{z}_{t}: t\in S\right\}.
\end{align*}
\end{dfn}
Similarly as in the real-valued case, the distribution function of the $\mathbb{C}^d$-valued random vector $\textbf{z}_t$ is the probability measure $P_{\textbf{z}_t}$ on $(\mathbb{C}^d,\mathscr{B}(\mathbb{C}^d))$ defined as,
\begin{align*}
P_{\textbf{z}_t} = \mathbb{P}\left[ \left\{ \omega \in \Omega \mid  \textbf{z}_t(\omega) \in B  \right\} \right], \quad \textnormal{for $B \in \mathscr{B}(\mathbb{C}^d)$.}
\end{align*}

\begin{dfn}
\label{def:weakstationarity}
The $\mathbb{C}^d$-valued process $\textbf{z}_\textnormal{\tiny{\textbullet}}  \coloneqq  \left(\textbf{z}_t\right)_{t\in \mathbb{N}} $ is weakly stationary if the components of $\textbf{z}_\textnormal{\tiny{\textbullet}} $ are square integrable, i.e.,
\begin{align*}
\mathbb{E}\left[ \left|z^{(k)}_t\right|^2 \right] < +\infty \quad \textnormal{$\forall k\in\{1,2,\ldots,d\}$, $\forall t\in \mathbb{N}$},
\end{align*}
and if for  any $\tau \in \left\{0\right\} \cup \mathbb{N}$ and for any pair  $\left(u,v\right) \in   \mathbb{N}\times   \mathbb{N}  $,
\begin{align*}
\mathbb{E}[ \textbf{z}_u ]  = \mathbb{E}[ \textbf{z}_v] \quad \textnormal{ and } \quad \ddot{\textbf{S}}_\tau[\textbf{z}_{u}] =  \ddot{\textbf{S}}_\tau[\textbf{z}_{v}] .
\end{align*}
\end{dfn}

We next consider finite sample estimators for location and autocovariance. Let $ \textbf{Z} \coloneqq  \left(\textbf{Z}_j \right)_{j\in \mathcal{T}}$, $\mathcal{T} = \left\{ 1,2,\ldots, T\right\}$, be a sampled  process generated by $\textbf{z}_\textnormal{\tiny{\textbullet}}$, where $\textbf{Z}_j$  is $\mathbb{C}^d$-valued for every $j \in \mathcal{T}$. We use the following classical finite sample estimators for the mean vector and the autocovariance matrix with lag $\tau \in \left\{0,1,\ldots,T-1\right\}$,
\begin{align*}
& \hat{\boldsymbol{\mu}}[{\textbf{Z}}] = \frac{1}{T}\sum_{j=1}^T \textbf{Z}_j,\\
&\tilde{\textbf{S}}_{\tau} \left[ \textbf{Z} \right] = \frac{1}{T-\tau}\sum_{j=1}^{T-\tau} \left(\textbf{Z}_j -  \hat{\boldsymbol{\mu}}[{\textbf{Z}}] \right)\left(\textbf{Z}_{j+\tau}-\hat{\boldsymbol{\mu}}[{\textbf{Z}}]\right)^\textnormal{H},
\end{align*}
where for $\tau =0$, the scaling $1/(T-1)$ is used for $\tilde{\textbf{S}}_{\tau}$. Furthermore, the symmetrized autocovariance matrix estimator with lag $\tau \in \left\{0,1,\ldots,T-1\right\}$ is defined as,
\begin{align*}
\hat{\textbf{S}}_\tau\left[ \textbf{Z}\right] = \frac{1}{2} \left[\tilde{\textbf{S}}_\tau\left[ \textbf{Z} \right]  + \left(\tilde{\textbf{S}}_\tau\left[ \textbf{Z} \right] \right)^\textnormal{H} \right].
\end{align*}
The symmetrized autocovariance matrix estimator is always conjugate symmetric. Note that the estimation of the eigenvectors of a conjugate symmetric matrix is more stable when compared to the estimation of the eigenvectors of a non-conjugate-symmetric matrix.

Let $\textbf{y}_1$ and $\textbf{y}_2$ be $\mathbb{R}^{d}$-valued random vectors such that the concatenated random vector, that is, $\begin{pmatrix}\textbf{y}_1^\top &\textbf{y}_2^\top\end{pmatrix}^\top$, follows a $\mathbb{R}^{2d}$-valued Gaussian distribution.
Then, the $\mathbb{C}^{d}$-valued random vector $\textbf{y} = \textbf{y}_1 +i \textbf{y}_2$ follows the $\mathbb{C}^{d}$-valued Gaussian distribution with the location parameter $\boldsymbol{\mu}_\textbf{y} $, the covariance matrix $\boldsymbol{\Sigma}_\textbf{y}$ and the relation matrix $\textbf{P}_\textbf{y}$, defined as follows,
\begin{align*}
&\boldsymbol{\mu}_\textbf{y} = \mathbb{E}\left[\textbf{y} \right] = \mathbb{E}\left[\textbf{y} _1\right] + i\mathbb{E}\left[ \textbf{y} _2 \right]\in \mathbb{C}^d,
\quad \boldsymbol{\Sigma}_\textbf{y} = \mathbb{E}\left[\left( \textbf{y}  - \boldsymbol{\mu}_\textbf{y}   \right) \left(\textbf{y} - \boldsymbol{\mu}_\textbf{y}  \right)^\textnormal{H} \right] \in \mathbb{C}^{d\times d},\\
& \textbf{P}_\textbf{y} = \mathbb{E}\left[\left( \textbf{y}  - \boldsymbol{\mu}_\textbf{y}   \right) \left(\textbf{y}  - \boldsymbol{\mu}_\textbf{y}  \right)^\top \right] \in \mathbb{C}^{d\times d}.
\end{align*}

We distinguish between complex- and real-valued Gaussian distributions with the number of given parameters --- $\mathcal{N}_d(\boldsymbol{\mu}_y,\boldsymbol{\Sigma})$ denotes a $d$-variate real-valued Gaussian distribution and $\mathcal{N}_d(\boldsymbol{\mu}_y,\boldsymbol{\Sigma},\textbf{P})$ denotes a $d$-variate complex-valued Gaussian distribution.

The classical central limit theorem (CLT) for independent and identically distributed complex-valued random variables is given as follows. Let $\{\textbf{x}_1,\textbf{x}_2,\ldots,\textbf{x}_n\}$ be a collection $\mathbb{C}^d$-variate  i.i.d. random vectors with square integrable components, mean $\boldsymbol{\mu}_{\textbf{x}}$, covariance matrix $\boldsymbol{\Sigma}$ and relation matrix $\textbf{P}$.  Then,
\begin{align*}
\frac{1}{\sqrt{n}}\sum_{j=1}^n\left(  \textbf{x}_j -  \boldsymbol{\mu}_{\textbf{x}} \right) \xrightarrow[n\rightarrow \infty]{\mathcal{D}} \textbf{x} \sim \mathcal{N}_d\left(\textbf{0},\boldsymbol{\Sigma},\textbf{P}\right),
\end{align*}
where $\xrightarrow[ ]{\mathcal{D}}$ denotes convergence in distribution.

We can harness the rich literature on real-valued limiting theorems by utilizing the following lemma.
\begin{lma}
\label{lma:rcconvergence}
Let $\{\textbf{v}_n\}_{n\in \mathbb{N}}$ be a collection of $\mathbb{R}^{2d}$-valued random vectors $\textbf{v}_j^\top = \begin{pmatrix}\textbf{x}_j^\top & \textbf{y}_j^\top \end{pmatrix}$, where $\textbf{x}_j, \textbf{y}_j$ are $\mathbb{R}^d$-valued for every $j \in \mathbb{N}$, and let $\textbf{v}^\top =\begin{pmatrix}\textbf{x}^\top & \textbf{y}^\top\end{pmatrix}$.  Then, $\textbf{v}_n  \xrightarrow[n\rightarrow \infty]{\mathcal{D}} \textbf{v}$ if and only if $\textbf{x}_n + i \textbf{y}_n  \xrightarrow[n\rightarrow \infty]{\mathcal{D}} \textbf{x} + i \textbf{y}  $.
\end{lma}
Lemma \ref{lma:rcconvergence} can be, for example, applied to Gaussian vectors in the following manner.
\begin{cor}
\label{cor:eqvrealcomplex}
Let $\{\textbf{v}_n\}_{n\in \mathbb{N}}$ be a collection of $\mathbb{R}^{2d}$-valued random vectors $\textbf{v}_j^\top = \begin{pmatrix}\textbf{x}_j^\top & \textbf{y}_j^\top\end{pmatrix}$, where $\textbf{x}_j, \textbf{y}_j $ are  $\mathbb{R}^d$-valued for every $j \in \mathbb{N}$. Let,
\begin{align*}
{\alpha_n}\sum_{j=1}^n\left(  \textbf{v}_j -  \boldsymbol{\mu}_{\textbf{v}} \right) \xrightarrow[n\rightarrow \infty]{\mathcal{D}} \textbf{v} \sim  \mathcal{N}_{2d}\left(\textbf{0},\boldsymbol{\Sigma}_\textbf{v}\right),
\end{align*}
where $\alpha_n \uparrow \infty$ as $n\rightarrow \infty$ and,
\begin{align*}
\boldsymbol{\mu}_{\textbf{v}} = \begin{pmatrix} \boldsymbol{\mu}_{\textbf{x}} \\ \boldsymbol{\mu}_{\textbf{y}}\end{pmatrix} \quad \textnormal{ and } \quad \boldsymbol{\Sigma}_\textbf{v} = \begin{pmatrix} \boldsymbol{\Sigma}_\textbf{x} &\boldsymbol{\Sigma}_\textbf{xy}  \\ \boldsymbol{\Sigma}_\textbf{xy}^\top  &\boldsymbol{\Sigma}_\textbf{y} \end{pmatrix}.
\end{align*}
Then, for the sequence of $\mathbb{C}^d$-valued random vectors $\textbf{z}_j = \textbf{x}_j + i \textbf{y}_j$, we have that,
\begin{align*}
{\alpha_n}\sum_{j=1}^n\left(  \textbf{z}_j -  \boldsymbol{\mu}_{\textbf{z}} \right) \xrightarrow[n\rightarrow \infty]{\mathcal{D}} \textbf{z} \sim \mathcal{N}_{d}\left(\textbf{0},\boldsymbol{\Sigma}_\textbf{z}, \textbf{P}_\textbf{z}\right),
\end{align*}
where $\boldsymbol{\mu}_{\textbf{z}} = \boldsymbol{\mu}_{\textbf{x}} + i \boldsymbol{\mu}_{\textbf{y}}$, $\boldsymbol{\Sigma}_\textbf{z} = \boldsymbol{\Sigma}_\textbf{x}+ \boldsymbol{\Sigma}_\textbf{y} + i(\boldsymbol{\Sigma}_\textbf{xy}^\top - \boldsymbol{\Sigma}_\textbf{xy})$ and $\textbf{P}_\textbf{z} = \boldsymbol{\Sigma}_\textbf{x}- \boldsymbol{\Sigma}_\textbf{y} + i(\boldsymbol{\Sigma}_\textbf{xy}^\top + \boldsymbol{\Sigma}_\textbf{xy})$
\end{cor}
Using the notation of Corollary \ref{cor:eqvrealcomplex}, we can straightforwardly demonstrate why we need two parameter matrices $\boldsymbol{\Sigma}_\textbf{z}$, $\textbf{P}_\textbf{z}$ and a single  parameter vector $\boldsymbol{\mu}_\textbf{z}$ in order to define a complex-valued Gaussian distribution. The vector $\boldsymbol{\mu}_\textbf{v}$ contains $2d$ free parameters of location and the block covariance matrix $\boldsymbol{\Sigma}_\textbf{v}$ contains $2d^2 + d$ free parameters of variances and covariances. We require that the information, regarding the location and the covariance structure, of the random vectors $\textbf{v}$ and $\textbf{z}$ is equal. Note that neither the covariance matrix $\boldsymbol{\Sigma}_\textbf{z}$ nor the relation matrix $\textbf{P}_\textbf{z}$ contains enough information to uniquely construct the covariance matrix $\boldsymbol{\Sigma}_\textbf{v}$. For example, by considering the real- and imaginary parts of $\boldsymbol{\Sigma}_\textbf{z}$, we get two matrix equations, which is not enough to define $\boldsymbol{\Sigma}_\textbf{x}$, $\boldsymbol{\Sigma}_\textbf{y}$ and $\boldsymbol{\Sigma}_\textbf{xy}$ uniquely. By defining both $\boldsymbol{\Sigma}_\textbf{z}$ and $\textbf{P}_\textbf{z}$, we get four matrix equations, which is enough to fix the free parameters.

In our work, we utilize  Lemma \ref{lma:rcconvergence} and Corollary \ref{cor:eqvrealcomplex} in order to apply real-valued results for complex-valued scenarios, but we wish to emphasize that they also work the other way --- the complex-valued asymptotic distribution automatically grants the corresponding real-valued limiting distribution.
For additional details regarding complex-valued statistics, see e.g.,  \citet{eriksson2010, goodman1963, picinbono1996}.

\section{Temporally uncorrelated components model}
\label{sec:bssmodel}
In this section, we consider a linear temporally uncorrelated components model for discrete time complex-valued processes. We formulate the so-called unmixing problem and give functionals that solve the corresponding problem. The temporally uncorrelated components mixing model is defined as follows.
\begin{dfn}
\label{model:BSS}
The $\mathbb{C}^d$-process $\textbf{x}_\textnormal{\tiny{\textbullet}} \coloneqq \left(\textbf{x}_t\right)_{t\in \mathbb{N}} $ follows the temporally uncorrelated components mixing model, if
\begin{align*}
\textbf{x}_t = \textbf{A} \textbf{z}_t + \boldsymbol{\mu}_{\textbf{x}}, \quad \forall t\in \mathbb{N} = \left\{1,2,\ldots \right\},
\end{align*}
where $\textbf{A} $ is a constant nonsingular $\mathbb{C}^{d\times d}$-matrix, $\boldsymbol{\mu}_{\textbf{x}} \in \mathbb{C}^d$ is a constant location parameter and $\textbf{z}_\textnormal{\tiny{\textbullet}}$  is a non-degenerate $\mathbb{C}^d$-process with components having continuous marginal distributions. In addition, the process $\textbf{z}_\textnormal{\tiny{\textbullet}}$  satisfies the following four conditions,
\begin{itemize}
\item[(1)] $\boldsymbol{\mu}_{\textbf{z}_t} = \textbf{0}, \quad \forall t \in \mathbb{N}$,
\item[(2)] $\textbf{S}_0\left[ \textbf{z}_t \right] = \textbf{I}_d \quad \forall t \in \mathbb{N}$,
\item[(3)] $\ddot{\textbf{S}}_\tau\left[ \textbf{z}_t \right] = \ddot{\boldsymbol{\Lambda}}_\tau \quad \forall t,\tau \in \mathbb{N}$,
\item[(4)] $\exists \tau \in \mathbb{N}$ : $\textbf{S}_\tau\left[ \textbf{z}_t \right] = \boldsymbol{\Lambda}_\tau = \textnormal{diag}\left( \lambda_\tau^{(1)},\ldots, \lambda_\tau^{(d)}\right),$
\end{itemize}
such that $+\infty >\lambda_\tau^{(1)} > \lambda_\tau^{(2)} > \cdots > \lambda_\tau^{(d)} > -\infty$.
\end{dfn}
To improve the fluency of the paper, from hereon, the term mixing model is used to refer to the temporally uncorrelated components mixing model. The conditions (1)-(4) of Definition \ref{model:BSS} imply that the latent process $\textbf{z}_\textnormal{\tiny{\textbullet}}$ is weakly stationary in the sense of Definition \ref{def:weakstationarity}. Hereby, in the population level, it suffices to consider the covariances and autocovariances for some generic $t \in \mathbb{N}$. Note that, under these model assumptions, the concatenated $\mathbb{R}^{2d}$-process $\begin{pmatrix} \textnormal{Re}(\textbf{z}_\textnormal{\tiny{\textbullet}}^\top) & \textnormal{Im}(\textbf{z}_\textnormal{\tiny{\textbullet}}^\top)  \end{pmatrix}^\top$ is not necessarily weakly stationary in the classical real-valued sense. One could also require that the concatenated vector is stationary, which can be done by adding a condition involving the complex-valued relation matrix, defined in Section \ref{sec:ctimeseries}, to the model. However, adding the extra condition would also fix the relationship, up to heterogeneous sign-changes, between the real- and imaginary parts inside a single component. In this paper, we do not make wish to make assumptions regarding the complex phase of a single component. Consequently, many of the following results involve a so called phase-shift matrix $\textbf{J}$, that is, a complex-valued diagonal matrix with diagonal entries of the form $\exp(i\theta_1), \ldots, \exp(i\theta_d)$. Note that a phase-shift matrix  $\textbf{J}$ is by definition unitary, i.e., $\textbf{J}\textbf{J}^\textnormal{H} =\textbf{J}^\textnormal{H} \textbf{J}  = \textbf{I}_d$.

In our approach, the existence of a single diagonal $\boldsymbol{\Lambda}_\tau$ suffices, thus making the model more general when compared to real-valued approaches where $\boldsymbol{\Lambda}_\tau$ is required to be diagonal for all $\tau \in\mathbb{N}$, see e.g., \cite{miettinen2012,miettinen2016separation}.

For a process $\textbf{x}_\textnormal{\tiny{\textbullet}}$ that follows the mixing model, the corresponding unmixing problem is to uncover the latent process $\textbf{z}_\textnormal{\tiny{\textbullet}}$, by using only the information contained in the observable process $\textbf{x}_\textnormal{\tiny{\textbullet}}$. Next, we define affine functionals that we refer to as solutions to the unmixing problem.
\begin{dfn}
\label{def:bsssolutions}
Let $\textbf{x}_\textnormal{\tiny{\textbullet}} \coloneqq \left(\textbf{x}_t\right)_{t\in \mathbb{N}}$ be process that satisfies Definition \ref{model:BSS}, such that condition (4) holds for some fixed $\tau$. The affine functional $g : \textbf{a} \mapsto \boldsymbol{\Gamma}(\textbf{a} - \boldsymbol{\mu})  :  \mathbb{C}^{d} \rightarrow \mathbb{C}^{d}$ is a solution to the corresponding unmixing problem if the process $g \circ \textbf{x}_\textnormal{\tiny{\textbullet}}$ satisfies conditions (1) - (3) and condition (4) is satisfied with the fixed $\tau$.
\end{dfn}

Given a process $\textbf{x}_\textnormal{\tiny{\textbullet}}$ that follows the mixing model and a corresponding solution $g : \textbf{a} \mapsto \boldsymbol{\Gamma}(\textbf{a} - \boldsymbol{\mu})  $, we refer to $\boldsymbol{A}$, see Definition \ref{model:BSS}, as the mixing matrix and we refer to $\boldsymbol{\Gamma}$ as the unmixing matrix. The objective of the unmixing matrix is to reverse the effects of the mixing matrix. However, under our model assumptions, we cannot uniquely determine the relationship between the mixing and the unmixing matrix. The relationship between the two matrices is considered in the following lemma.
\begin{lma}
\label{lma:permphase}
Let $\textbf{x}_\textnormal{\tiny{\textbullet}} \coloneqq \left(\textbf{x}_t\right)_{t\in \mathbb{N}}$ be a process that satisfies Definition \ref{model:BSS}. The functional $g : \textbf{a} \mapsto \boldsymbol{\Gamma}(\textbf{a} - \boldsymbol{\mu})  :  \mathbb{C}^{d} \rightarrow \mathbb{C}^{d}$ is a solution to the corresponding unmixing problem if and only if,
\begin{align*}
\boldsymbol{\mu} = \boldsymbol{\mu}_{\textbf{x}}  \quad \textnormal{ and } \quad \boldsymbol{\Gamma}\textbf{A} = \textbf{J},
\end{align*}
where $\textbf{J}\in \mathbb{C}^{d\times d}$ is a phase-shift matrix. 
\end{lma}

The real-valued version of Lemma \ref{lma:permphase} is achieved by considering  phase-shift matrices with diagonal elements $\exp(i\theta_1), \ldots, \exp(i\theta_d)$, where $\theta_j$ are multiples of $\pi$. Consequently, in the real-valued case, the matrix $\textbf{J}$ is a so-called sign-change matrix. Note that all possible phase-shift matrices can be constructed by considering phases on some $2\pi$-length interval, e.g., the usual $[0,2\pi)$.

By Lemma \ref{lma:permphase}, we cannot distinguish between solutions that contain unmixing matrices that are a phase-shift away from each other. Thus, we say that two solutions $g_1 : \textbf{a} \mapsto \boldsymbol{\Gamma}_1(\textbf{a} - \boldsymbol{\mu}_1) $ and $g_2 : \textbf{a} \mapsto \boldsymbol{\Gamma}_2(\textbf{a} - \boldsymbol{\mu}_2) $ are equivalent, if there exists a phase-shift matrix $\textbf{J}$ such that $\boldsymbol{\Gamma}_1 = \boldsymbol{\Gamma}_2 \textbf{J}$. On the contrary, if $g_1$ and $g_2$ are both solutions to the same unmixing problem, it follows that $\boldsymbol{\mu}_1 = \boldsymbol{\mu}_2$.

We next provide a solution procedure for the unmixing problem. Recall that in the mixing model, we assume that the mixing matrix $\textbf{A}$ and the latent process $\textbf{z}_\textnormal{\tiny{\textbullet}}$ are unobservable. Therefore, the solution procedure relies solely on statistics of the observable process $\textbf{x}_\textnormal{\tiny{\textbullet}}$.

\begin{thm}
\label{bss:solutions}
Let $\textbf{x}_\textnormal{\tiny{\textbullet}} \coloneqq \left(\textbf{x}_t\right)_{t\in \mathbb{N}}$ be a process that satisfies Definition \ref{model:BSS}. The functional $g : \textbf{a} \mapsto \boldsymbol{\Gamma}(\textbf{a} - \boldsymbol{\mu})  :  \mathbb{C}^{d} \rightarrow \mathbb{C}^{d}$ is a solution to the corresponding unmixing problem if and only if the following eigenvector-equation,
\begin{align}
&\left( \textbf{S}_0\left[ \textbf{x}_t \right] \right)^{-1}  \textbf{S}_\tau \left[ \textbf{x}_t \right]\boldsymbol{\Gamma}^\textnormal{H}  = \boldsymbol{\Gamma}^\textnormal{H} \boldsymbol{\Lambda}_\tau, \quad \forall t \in \mathbb{N}, \label{eq:sol1}
\end{align}
and the following scaling-equation,
\begin{align}
&\boldsymbol{\Gamma} \textbf{S}_0\left[\textbf{x}_t\right] \boldsymbol{\Gamma}^\textnormal{H} = \textbf{I}_d,\quad \forall t \in \mathbb{N}
\label{eq:sol2}
\end{align}
are satisfied.
\end{thm}
One can apply Theorem \ref{bss:solutions} in order to find solutions by  first estimating the eigenvalues and -vectors of $\left( \textbf{S}_0\left[ \textbf{x}_t \right] \right)^{-1}  \textbf{S}_\tau \left[ \textbf{x}_t \right]$ and then scaling the eigenvectors such that the scaling equation is satisfied. In addition, the eigenvectors should be ordered such that the corresponding eigenvalues, i.e., the diagonal elements of $\boldsymbol{\Lambda}_\tau$ are in decreasing order. Such solutions can be generated straightforwardly with the following corollary.
\begin{cor}
\label{cor:solution}
Let $\textbf{x}_\textnormal{\tiny{\textbullet}} \coloneqq \left(\textbf{x}_t\right)_{t\in \mathbb{N}}$ be a process that satisfies Definition \ref{model:BSS} and denote $\textbf{S}_j \coloneqq \textbf{S}_j[\textbf{x}_t]$. Let $\textbf{S}_0^{-1/2}$ be a conjugate symmetric matrix, such that $\textbf{S}_0^{-1/2}\textbf{S}_0\textbf{S}_0^{-1/2} = \textbf{I}_d$. Then, the eigendecomposition $\textbf{S}_0^{-1/2} \textbf{S}_\tau \textbf{S}_0^{-1/2} = \textbf{V} \boldsymbol{\Lambda}_\tau \textbf{V}^\textnormal{H}$, is satisfied for some unitary $\textbf{V}$ and the functional
\begin{align*}
 g : \textbf{a} \mapsto \textbf{V}^\textnormal{H}\textbf{S}_0^{-\frac{1}{2}} (\textbf{a} - \boldsymbol{\mu}_{\textbf{x}})  :  \mathbb{C}^{d} \rightarrow \mathbb{C}^{d},
\end{align*}
is a solution to the corresponding unmixing problem.
\end{cor}
The solutions for the unmixing problem are affine invariant in the following sense.
\begin{lma}
\label{lma:affeqv}
Let $\textbf{x}_\textnormal{\tiny{\textbullet}} \coloneqq \left(\textbf{x}_t\right)_{t\in \mathbb{N}}$ be a process that satisfies Definition \ref{model:BSS} and let $g : a \mapsto \boldsymbol{\Gamma}(\textbf{a} - \boldsymbol{\mu})$ be a corresponding solution. Furthermore, let $\tilde{\textbf{x}_\textnormal{\tiny{\textbullet}}}\coloneqq \left(\textbf{C}\textbf{x}_t + \textbf{b}\right)_{t\in \mathbb{N}}$, where  $\textbf{C} \in \mathbb{C}^{d\times d}$ is nonsingular, $\textbf{b} \in \mathbb{C}^d$ and let  $\tilde{g} : a \mapsto \tilde{\boldsymbol{\Gamma}} ( \textbf{a} - \tilde{\boldsymbol{\mu}})$ be a solution for the process $\tilde{\textbf{x}_\textnormal{\tiny{\textbullet}}}$. Then,
\begin{align*}
\tilde{\boldsymbol{\Gamma}}  = \textbf{J}{\boldsymbol{\Gamma}}\textbf{C}^{-1},
\end{align*}
for some  phase-shift matrix $\textbf{J}$.
\end{lma}

\section{Estimation and asymptotic properties}
\label{sec:asympBSS}

The algorithm for multiple unknown signals extraction (AMUSE), \cite{tong1990amuse}, is a  widely applied blind source separation unmixing procedure. The AMUSE procedure can be applied to solve the unmixing problem presented in this paper. AMUSE and the corresponding asymptotic properties have been previously studied in the real-valued case, see e.g, \cite{miettinen2012}. Here, we adopt the estimation part of the AMUSE algorithm. However, our underlying model assumptions are not identical to the ones in  \cite{miettinen2012} and \cite{tong1990amuse}. Thus, in order to avoid misinterpretations, we refrain from using the term AMUSE and we simply use the term unmixing procedure.
In this section, we consider the consistency and the limiting distribution of the corresponding unmixing estimator under complex-valued long-range and short-range dependent processes, significantly extending the results given in \cite{miettinen2012}. 

Analysis of the temporally uncorrelated components relies on estimating the mean vector $\boldsymbol{\mu}$ and the unmixing matrix $\boldsymbol{\Gamma}$. Our main focus is on the asymptotic theory of the unmixing matrix estimator, since under the mixture model of this paper, the estimation of the location parameter is  straightforward.  The following definition presents finite sample solutions for the unmixing problem.
\begin{dfn}
\label{def:estimating}
Let $\textbf{x}_\textnormal{\tiny{\textbullet}} \coloneqq \left(\textbf{x}_t\right)_{t\in \mathbb{N}}$ be a process that satisfies Definition \ref{model:BSS} and let $\textbf{X} $ be a $ \mathbb{C}^{T \times d}$-valued, $1\leq d < T < \infty$, sampled stochastic process generated by $\textbf{x}_\textnormal{\tiny{\textbullet}}$. Let $ \hat{\boldsymbol{\mu}} \coloneqq \hat{\boldsymbol{\mu}}[\textbf{X}] $ be the sample mean and let $\textbf{1}_T$ be a $\mathbb{R}^T$-vector full of ones.  Then, the mapping $\hat{g}: \textbf{C} \mapsto (\textbf{C} - \textbf{1}_T\hat{\boldsymbol{\mu}}^\top ) \hat{\boldsymbol{\Gamma}}^\top : \mathbb{C}^{T \times d} \rightarrow \mathbb{C}^{T \times d}$ is a solution to the finite sample unmixing problem, if
\begin{align*}
\hat{\boldsymbol{\Gamma} }  \hat{\textbf{S}}_0[\textbf{X} ] \hat{\boldsymbol{\Gamma}}^\textnormal{H}  =\textbf{I}_d \qquad \textnormal{ and } \qquad \hat{\boldsymbol{\Gamma} }  \hat{\textbf{S}}_\tau [\textbf{X} ] \hat{\boldsymbol{\Gamma}}^\textnormal{H}  =\hat{\boldsymbol{\Lambda}}_\tau =   \textnormal{diag}\left( \hat{\lambda}_\tau^{(1)},\ldots, \hat{\lambda}_\tau^{(d)}\right),
\end{align*}
where $\hat{\lambda}_\tau^{(1)} \geq  \hat{\lambda}_\tau^{(2)} \geq \ldots \geq  \hat{\lambda}_\tau^{(d)}.$
\end{dfn}
Recall, that $\hat{\textbf{S}}_\tau$ is the symmetrized autocovariance matrix, and it is, by definition, conjugate symmetric. In addition, recall that the diagonal entries and the eigenvalues of a conjugate symmetric matrix are, again by definition, real-valued. Hereby, in Definition \ref{def:estimating}, the ordering of $ \hat{\lambda}_\tau^{(j)}$ is always well-defined.

Let $\textbf{X}_j$ denote the transpose of the $j$th row of $\textbf{X}$, that is, the value the sampled $\mathbb{C}^d$-process takes at time $j$. Given a finite sample solution $\hat{g}: \textbf{C} \mapsto (\textbf{C} - \textbf{1}_T\hat{\boldsymbol{\mu}}^\top ) \hat{\boldsymbol{\Gamma}}^\top$, we get an estimate for the latent process $\textbf{z}_\textnormal{\tiny{\textbullet}}$ from the mapping $\hat{g}(\textbf{X})$. Note that, under the mapping,
\begin{align*}
\hat{g}(\textbf{X}) =  \left(\textbf{X} - \textbf{1}_T\hat{\boldsymbol{\mu}}^\top\right) \hat{\boldsymbol{\Gamma}}^\top,
\end{align*}
we get that  $\textbf{X}_j$ is mapped as follows,
\begin{align}
\label{eq:rowgamma}
\hat{\boldsymbol{\Gamma}}\left(\textbf{X}_j - \hat{\boldsymbol{\mu}} \right).
\end{align}
Hereby, Equation \eqref{eq:rowgamma} corresponds to the population version given in Definition \ref{def:bsssolutions}.

As in the population case, we can solve the finite sample unmixing problem by utilizing the following lemma.
\begin{lma}
\label{lma:samplesol}
Let $\textbf{x}_\textnormal{\tiny{\textbullet}} \coloneqq \left(\textbf{x}_t\right)_{t\in \mathbb{N}}$ be a process that satisfies Definition \ref{model:BSS} and let $\textbf{X}$ be a $ \mathbb{C}^{T \times d}$-valued, $1\leq d < T < \infty$, sampled stochastic process generated by $\textbf{x}_\textnormal{\tiny{\textbullet}}$. Denote $ \hat{\boldsymbol{\mu}} \coloneqq \hat{\boldsymbol{\mu}}[\textbf{X}]$ and $\hat{\textbf{S}_j}\coloneqq \hat{\textbf{S}_j}[\textbf{X}]$, $j \in \{0,\tau\}$. Let $\hat{\boldsymbol{\Sigma}}_0$ be a conjugate symmetric matrix that satisfies $\hat{\boldsymbol{\Sigma}}_0\hat{\textbf{S}_0}\hat{\boldsymbol{\Sigma}}_0 = \textbf{I}_d$ and let  $\hat{\boldsymbol{\Sigma}}_0 \hat{\textbf{S}_\tau} \hat{\boldsymbol{\Sigma}}_0= \xhat{\textbf{V}} \hat{\boldsymbol{\Lambda}}_\tau \xhat{\textbf{V}}^\textnormal{H}$, where $\xhat{\textbf{V}}$ is unitary. Such matrices $\hat{\boldsymbol{\Sigma}}_0$, $\xhat{\textbf{V}}$ exists, and the mapping
\begin{align*}
\hat{g}: \textbf{C} \mapsto (\textbf{C} - \textbf{1}_T\hat{\boldsymbol{\mu}}^\top ) \left(\xhat{\textbf{V}}^\textnormal{H}\hat{\boldsymbol{\Sigma}}_0  \right)^\top : \mathbb{C}^{T \times d} \rightarrow \mathbb{C}^{T \times d},
\end{align*}
is a solution to the finite sample unmixing problem.
\end{lma}
Using Lemma \ref{lma:samplesol}, we can straightforwardly implement the unmixing procedure. The first step is to calculate $\hat{\textbf{S}}_0 \coloneqq \hat{\textbf{S}}_0[\textbf{X}]$, from a realization $\textbf{X}$, and the corresponding conjugate symmetric inverse square root $\hat{\boldsymbol{\Sigma}}_0 \coloneqq \hat{\textbf{S}}_0^{-1/2}$. Recall that, by assumption, the components of $\textbf{z}_\textnormal{\tiny{\textbullet}} $ have continuous marginal distributions and the mixing matrix is nonsingular. Thus, the eigenvalues of covariance matrix estimates are always real-valued and almost surely positive. Hereby, the matrix $\hat{\boldsymbol{\Sigma}}_0$ can be obtained conveniently by estimating the eigendecomposition of $\hat{\textbf{S}}_0$. The next step is to choose a lag-parameter $\tau$ and estimate the eigendecomposition $\hat{\textbf{S}}_\tau[ \textbf{X}\hat{\boldsymbol{\Sigma}}_0^\top ] =\hat{\boldsymbol{\Sigma}}_0\hat{\textbf{S}}_\tau[ \textbf{X}] \hat{\boldsymbol{\Sigma}}_0=  \hat{\textbf{V}} \hat{\boldsymbol{\Lambda}}_\tau \hat{\textbf{V}}^\textnormal{H} $. The covariance matrix and autocovariance matrix estimators are affine equivariant in the sense that $\hat{\textbf{S}}_j[\textbf{X} \textbf{C}^\top] =\textbf{C}\hat{\textbf{S}}_j[\textbf{X}]\textbf{C}^\textnormal{H}$, $j \in\{0,\tau\}$, for all nonsingular $\mathbb{C}^{d\times d}$-matrices $\textbf{C}$. An estimate for the latent process can then be found via the mapping $(\textbf{X} - \textbf{1}_T\hat{\boldsymbol{\mu}}^\top ) (\hat{\textbf{V}}^\textnormal{H}\hat{\boldsymbol{\Sigma}}_0  )^\top,$ where $\hat{\boldsymbol{\mu}}$ is the sample mean vector of $\textbf{X}$.

In practice, one should choose the lag parameter $\tau \neq 0$ such that the diagonal elements of $\hat{\boldsymbol{\Lambda}}_\tau$ are as distinct as possible. Strategies for choosing $\tau$ are discussed, for example, in \cite{cichocki2002adaptive}. From a computational perspective, the estimation procedure is relatively fast. Hereby, in practice, one should try several different values of $\tau$ and study the diagonal elements of $\hat{\boldsymbol{\Lambda}}_\tau$. It is usually beneficial to start with lag parameters close to $\tau=1$, as in several applications the autocovariances tend to diminish as $\tau$ grows. In addition, values for $\tau$ that are close to $T$ should be avoided, since one might end up estimating the autocovariance matrix using only a small number of observations.

We emphasize that one can apply the theory of this paper to other estimators as well. That is, under minor model assumptions, the estimators $\hat{\textbf{S}}_0$,  $\hat{\textbf{S}}_\tau$ can be replaced with any matrix valued estimators that have the so-called complex-valued affine equivariance property, see \cite{ilmonen2013}.

The conditions given in Definition \ref{def:estimating} remain true, if we replace $\hat{\boldsymbol{\Gamma}}$ with $\textbf{J}\hat{\boldsymbol{\Gamma}}$, where $\textbf{J} \in \mathbb{C}^{d\times d}$ can be any phase-shift matrix. In other words, if $\hat{\boldsymbol{\Gamma}}$ is an unmixing matrix estimate, then $\textbf{J}\hat{\boldsymbol{\Gamma}}$ is an equally viable estimate. Thus, as in the population level, we say that the estimates $\hat{\boldsymbol{\Gamma}}_1$ and $\hat{\boldsymbol{\Gamma}}_2$ are equivalent if $\hat{\boldsymbol{\Gamma}}_1 = \textbf{J}\hat{\boldsymbol{\Gamma}}_2$ for some phase-shift matrix $\textbf{J}$. Furthermore, similarly as in Lemma \ref{lma:affeqv}, we have that the unmixing matrix estimates are affine invariant in the following sense.
\begin{lma}
\label{lma:affinvsample}
Let $\textbf{x}_\textnormal{\tiny{\textbullet}} \coloneqq \left(\textbf{x}_t\right)_{t\in \mathbb{N}}$ be a process that satisfies Definition \ref{model:BSS} and let $\textbf{X} $ be a $ \mathbb{C}^{T \times d}$-valued, $1\leq d < T < \infty$, sampled stochastic process generated by $\textbf{x}_\textnormal{\tiny{\textbullet}}$ and let $\hat{g} : \textbf{C} \mapsto (\textbf{C} - \textbf{1}_T \hat{\boldsymbol{\mu}}^\top) \hat{\boldsymbol{\Gamma}}^\top$ be a corresponding finite sample solution defined in Lemma \ref{lma:samplesol}. Let $\textbf{B} \in \mathbb{C}^{d\times d}$ be nonsingular, let $\textbf{b} \in \mathbb{C}^d$, let $\xtilde{\textbf{X}} = \textbf{X}\textbf{B}^\top + \textbf{1}_T\textbf{b}^\top $ and let $\tilde{g} : \textbf{C} \mapsto (\textbf{C} - \textbf{1}_T \tilde{\boldsymbol{\mu}}^\top) \xtilde{\boldsymbol{\Gamma}}^\top$ be a corresponding finite sample solution for $\xtilde{\textbf{X}}$. Then, $\tilde{\boldsymbol{\Gamma}} = {\textbf{J}} {\hat{\boldsymbol{\Gamma}}}\textbf{B}^{-1}$ is a finite sample solution for $\xtilde{\textbf{X}}$, where $\textbf{J}$ is some phase-shift matrix.
\end{lma}
Justified by the affine invariance property given in Lemmas \ref{lma:affeqv} and  \ref{lma:affinvsample}, we can, without loss of generality, derive the rest of the theory under the assumption of trivial mixing, that is, the case when the mixing matrix is $\textbf{A} = \textbf{I}_d$. See also the beginning of the Proof of Lemma \ref{lma:consistent}. 

We next consider limiting properties of the finite sample solutions. Note that, the finite sample statistics and the sampled stochastic process $\textbf{X}$ all depend on the sample size $T$. Furthermore, we mainly follow the multivariate probabilistic Bachmann-Landau notation presented in \cite{van2000}. We use the short expression $\textbf{X}_{T} = o_p(1)$ to denote that a sequence of $\mathbb{C}^{d\times d}$-matrices $\textbf{X}_{1},\textbf{X}_{2}\ldots,$ converges in probability to a zero matrix as $T\rightarrow \infty$.  In addition, we use $\textbf{Y}_{\beta} = \mathcal{O}_p(1)$ to denote that a collection of $\mathbb{C}^{d\times d}$-matrices $\{\textbf{Y}_{\beta} : \beta\in B \}$ is uniformly tight, i.e., bounded in probability under some non-empty indexing set $B$. By saying that a matrix is uniformly tight, we mean that every element of the matrix is uniformly tight. Note that in the framework of this paper, our usage of the notation is equivalent with the one presented in  \cite{van2000}.

\begin{lma}
\label{lma:consistent}
Let $\textbf{x}_\textnormal{\tiny{\textbullet}} \coloneqq \left(\textbf{x}_t\right)_{t\in \mathbb{N}}$ be a process that satisfies Definition \ref{model:BSS} and let $\textbf{X}$ be a $ \mathbb{C}^{T \times d}$-valued, $1\leq d < T < \infty$, sampled stochastic process generated by $\textbf{x}_\textnormal{\tiny{\textbullet}}$ and  let $\hat{g} : \textbf{C} \mapsto (\textbf{C} - \textbf{1}_T \hat{\boldsymbol{\mu}}^\top) \hat{\boldsymbol{\Gamma}}^\top$ be a $T$-indexed sequence of corresponding finite sample solutions defined in Lemma \ref{lma:samplesol}. Furthermore, let $\alpha_T(\hat{\textbf{S}_0}[\textbf{X}] - {\textbf{S}}_0[\textbf{x}_t]) = \mathcal{O}_p(1)$ and $\beta_T(\hat{\textbf{S}_\tau}[\textbf{X}] - {\textbf{S}}_\tau[\textbf{x}_t]) = \mathcal{O}_p(1)$, for some real-valued sequences $\alpha_T$ and $\beta_T$, which satisfy $\alpha_T \uparrow \infty$, $\beta_T \uparrow \infty$ as $T \rightarrow \infty$. Then, there exists a  sequence of $T$-indexed phase-shift matrices $\xhatt{\textbf{J}}$, such that the following holds asymptotically,
\begin{align*}
\gamma_T \left( \xhatt{\textbf{J}}\hat{\boldsymbol{\Gamma}} -  \boldsymbol{\Gamma} \right) = \mathcal{O}_p(1),
\end{align*}
where $\gamma_t = \min(\alpha_t,\beta_t)$, $\forall t\in\mathbb{N}$.
\end{lma}
Under the assumptions of Lemma \ref{lma:consistent}, we have that the sequence of unmixing matrix estimators is consistent in the sense that there exists a sequence of $T$-indexed phase-shift matrices $\hat{\textbf{J}}$ such that $\hat{\textbf{J}}\hat{\boldsymbol{\Gamma}} $ converges in probability to $\boldsymbol{\Gamma}$, as $T\rightarrow \infty$. This convergence in probability follows directly from Lemma \ref{lma:inverse}.

\begin{thm}
\label{thm:limiting}
Let $\textbf{X}$, $g$ and $\hat{g}$ be defined as in Lemma \ref{lma:consistent} and denote the element $(j,k)$ of  $ \hat{\textbf{S}_\tau}[\textbf{X}]$ as  $[\hat{\textbf{S}_\tau}]_{jk}$. Then, under the assumptions of Lemma \ref{lma:consistent} and the  trivial mixing scenario $\textbf{A} = \textbf{I}_d$, there exists a sequence of $T$-indexed phase-shift matrices $\xhatt{\textbf{J}}$, such that,
\begin{align*}
&\gamma_T\left(\xhatt{\textbf{J}}_{jj}\hat{\boldsymbol{\Gamma}}_{jj} - 1\right) = \frac{\gamma_T}{2}\left( \left[\hat{\textbf{S}_0}\right]_{jj} - 1 \right) +\mathcal{O}_p(1/\gamma_T), \quad \forall j\in\{1,\ldots,d\} \quad\textnormal{ and}\\
&\gamma_T\left(\lambda^{(k)}_\tau- \lambda^{(j)}_\tau\right) \xhatt{\textbf{J}}_{jj} \hat{\boldsymbol{\Gamma}}_{jk} =\gamma_T\left( \lambda^{(j)}_\tau \left[\hat{\textbf{S}_0}\right]_{jk} - \left[\hat{\textbf{S}_\tau}\right]_{jk}\right) + \mathcal{O}_p(1/\gamma_T), \quad j\neq k.
\end{align*}
\end{thm}
By Theorem \ref{thm:limiting}, we can directly find the asymptotic distribution of the unmixing matrix estimator $\hat{\boldsymbol{\Gamma}}$, if we have the asymptotic distributions, and the convergence rates, of the estimators $\hat{\textbf{S}}_0$ and $\hat{\textbf{S}}_\tau$. Note that, if the rates $\alpha_T$ and $\beta_T$, given in Lemma \ref{lma:consistent}, differ, the estimator with the faster convergence will tend to zero in probability in Theorem \ref{thm:limiting}. Furthermore, note that the asymptotic distributions of the diagonal elements of $\gamma_T(\hat{\textbf{J}}\hat{\boldsymbol{\Gamma}} -\textbf{I}_d)$ only depend on the asymptotic distribution of the covariance matrix estimator $\hat{\textbf{S}}_0$. However, the rate of convergence of the off-diagonal elements depends on both $\hat{\textbf{S}}_0$ and $\hat{\textbf{S}}_\tau$, and consequently, if the covariance estimator $\hat{\textbf{S}}_0$ converges faster than the autocovariance estimator $\hat{\textbf{S}}_\tau$, the diagonal elements of   $\gamma_T(\hat{\textbf{J}}\hat{\boldsymbol{\Gamma}} -\textbf{I}_d)$ converge to zero in distribution and in probability. Conversely, it may happen that the autocovariance estimator $\hat{\textbf{S}}_\tau$ converges faster. In that case, the off-diagonal elements of   $\gamma_T(\hat{\textbf{J}}\hat{\boldsymbol{\Gamma}} -\textbf{I}_d)$ converge to zero in distribution and in probability. Examples are given in Subsection \ref{sec:noncentral}.

We again want to highlight that, aside from some minor model assumptions, affine equivariance is the only property that we require from the matrix valued estimators, and the estimators can be directly replaced with some other affine equivariant estimators. The estimation procedure could, for example, be robustified and Theorem \ref{thm:limiting} would directly give the asymptotic behavior of the robustified alternative, given that the asymptotic behavior of the robust estimators is known.

We can present Theorem \ref{thm:limiting} in matrix form as follows,
\begin{align*}
&\gamma_T \left( \textnormal{diag}\left[ \hat{\textbf{J}}\hat{\boldsymbol{\Gamma}} - \textbf{I}_d \right] \right) = \gamma_T\left(  \frac{1}{2}  \textnormal{diag} \left[\hat{\textbf{S}}_0 - \textbf{I}_d \right] \right) + \mathcal{O}_p(1/\gamma_T),\\
&\gamma_T \left( \hat{\textbf{J}}\hat{\boldsymbol{\Gamma}} - \textnormal{diag}\left[ \hat{\textbf{J}}\hat{\boldsymbol{\Gamma}}\right] \right) = \gamma_T\left(  \textbf{H}  \odot \left[\boldsymbol{\Lambda}_\tau\hat{\textbf{S}}_0 - \hat{\textbf{S}}_\tau   \right] \right) + \mathcal{O}_p(1/\gamma_T),
\end{align*}
where $\textbf{H}_{jj} = 0$ and $\textbf{H}_{jk} = (\lambda_\tau^{(k)} - \lambda_\tau^{(j)})^{-1}$, $j\neq k$, $\odot$ denotes the Hadamard, i.e., the entrywise product  and $\hat{\textbf{J}}$ is the $T$-indexed sequence of phase-shift matrices that set the diagonal elements of $\hat{\boldsymbol{\Gamma}}$ to be on the positive real-axis.

Theorem  \ref{thm:limiting} can be applied under any nonsingular mixing. Let $g$ and $\hat{g}$ be a solution and a finite sample solution, respectively, under the nonsingular mixing matrix $\textbf{A}$ and denote the corresponding unmixing matrix as $\boldsymbol{\Gamma}$ and denote the corresponding sample unmixing matrix estimator as $\hat{\boldsymbol{\Gamma}}$. Likewise, let $g_0$ and $\hat{g}_0$ be a solution and a finite sample solution, respectively, under the trivial mixing matrix and denote the corresponding sample unmixing matrix estimate as $\hat{\boldsymbol{\Gamma}}_0$. The following relation can then be applied to generalize  results under the trivial mixing to any nonsingular mixing scenario,
\begin{align*}
\gamma_T\left(\textnormal{vec}\left[\hat{\textbf{J}}\hat{\boldsymbol{\Gamma}} -  \boldsymbol{\Gamma} \right]\right) =\gamma_T\left(\textbf{A}^{-\top} \otimes \textbf{I}_d \right)\left(\textnormal{vec}\left[\hat{\textbf{J}}_0\hat{\boldsymbol{\Gamma}}_0 -  \textbf{I}_d \right]\right),
\end{align*}
where $\textbf{A}^{-\top} $ is the inverse of the transpose, $\otimes$ denotes the Kronecker product,  $\hat{\textbf{J}}$ is a $T$-indexed sequence of phase-shift matrices that set the phases of the diagonal elements of $\hat{\boldsymbol{\Gamma}}$ and $\boldsymbol{\Gamma}$ to be equal and  $\hat{\textbf{J}}_0$ is a $T$-indexed sequence of phase-shift matrices that set the phases of the diagonal elements of $\hat{\boldsymbol{\Gamma}}_0$ to be on  the positive real-axis.

In the following subsections, we present examples of convergence of the autocovariance and covariance matrix estimators under a broad class of stochastic processes, notably, including long-range dependent processes.

\subsection{Limiting distributions under summable covariance structures}\label{sec:bmajor}
In this section, we consider a class of stochastic processes that satisfy the Breuer-Major Theorem for weakly stationary processes, which we from hereon refer to as the Breuer-Major Theorem. The Breuer-Major Theorem is  considered in the context of Gaussian subordinated processes, see e.g. \cite{breuer-major, arcones1994limit, ivan-et-al}. A univariate real-valued Gaussian subordinated process ${z}_\textnormal{\tiny{\textbullet}}$, defined on a probability space $(\Omega,\mathcal{F},\mathbb{P})$, is a weakly stationary process that can be expressed in the form ${z}_\textnormal{\tiny{\textbullet}} = f \circ \textbf{y}_\textnormal{\tiny{\textbullet}}$, where $\textbf{y}_\textnormal{\tiny{\textbullet}}$ is a $\mathbb{R}^\ell$-variate Gaussian process and $f : \mathbb{R}^\ell \rightarrow \mathbb{R}$ is a measurable function.

We emphasize that Gaussian subordinated processes form a very rich model class. For example, recently in \cite{viitasaari2017}, the authors showed that arbitrary one dimensional marginal distributions and a rich class of covariance structures can be modeled by $f \circ \textbf{y}_\textnormal{\tiny{\textbullet}}$ with simple univariate stationary Gaussian process $\textbf{y}_\textnormal{\tiny{\textbullet}}$. While the model class is very rich, underlying driving Gaussian processes still provide a large toolbox for analyzing limiting behavior of various estimators. Such limit theorems have been a topic of active research for decades. For recent relevant papers on the topic, we refer to \cite{ivan-et-al,nou-nua-funcBM,nou-pec2,nou-pec1,n-p} and the references therein.

We next give the definitions of Hermite polynomials and Hermite ranks. We define the $k$th, $k \in \{0\}\cup \mathbb{N}$, (probabilistic) Hermite polynomial $H_k$, using Rodrigues' formula,
\begin{align*}
H_k(x) = (-1)^k \exp\left({x^2}/{2}\right) \frac{d^k}{dx^k}\exp\left(-{x^2}/{2}\right).
\end{align*}
The first four Hermite polynomials are hereby $H_0(x) = 1$, $H_1(x) = x$, $H_2(x) = x^2-1$ and $H_3(x) = x^3 - 3x$. The set $\{H_k/\sqrt{k!} : k \in\{0\}\cup \mathbb{N}\}$ forms an orthonormal basis on the Hilbert-space $\mathcal{L}^2(\mathbb{R}, P_y)$, where $P_y$  denotes the law of a univariate standard Gaussian random variable. Consequently, every function $f \in \mathcal{L}^2(\mathbb{R}, P_y)$ can be decomposed as,
\begin{align}
\label{eq:hermitedecomp}
f(x) = \sum_{k=0}^\infty a_k H_k(x),
\end{align}
where $a_k \in\mathbb{R}$ for every $k\in\{0\}\cup \mathbb{N}$. If $x$ and $y$ follow the univariate standard Gaussian distribution, the orthogonality of the Hermite polynomials and the decomposition of Equation \eqref{eq:hermitedecomp} give,
\begin{align*}
\mathbb{E}\left[f(x) f(y) \right] = \sum_{k=0}^\infty k! \alpha_k^2\left( \mathbb{E}\left[(x - \mathbb{E}[x] )(y - \mathbb{E}[y]) \right] \right)^k = \sum_{k=0}^\infty k! \alpha_k^2\left( \textnormal{Cov}[x,y]\right)^k.
\end{align*}
The Hermite rank for a function $f$ is defined as follows.
\begin{dfn}[Hermite rank]
Let $\textbf{y}$ be a $\mathbb{R}^\ell$-valued Gaussian random vector, $\ell \in\mathbb{N}$, and let $f: \mathbb{R}^\ell \rightarrow \mathbb{R}$ be a function such that $f\circ\textbf{y}$ is square-integrable. The function $f$ has Hermite rank $q$, with respect to $\textbf{y}$, if
\begin{align*}
\mathbb{E}\left[ \left( f\left(\textbf{y}\right) - \mathbb{E}\left[f \left(\textbf{y}\right) \right]\right)p_m\left(\textbf{y}\right)\right]=0,
\end{align*}
for all polynomials $p_m: \mathbb{R}^d \rightarrow \mathbb{R}$ that are of degree $m \leq q-1$ and there exists a polynomial $p_q$ of degree $q$ such that
\begin{align*}
\mathbb{E}\left[ \left( f\left(\textbf{y}\right) - \mathbb{E}\left[f \left(\textbf{y}\right) \right]\right)p_q\left(\textbf{y}\right)\right]\neq 0.
\end{align*}
\end{dfn}
Note that, in one dimensional setting, the Hermite rank $q$ of a function $f$ is the smallest non-negative integer in Equation \eqref{eq:hermitedecomp}, such that $\alpha_q \neq 0$.  We next present the multivariate version of the well-known Breuer-Major Theorem  from \cite{breuer-major}.
\begin{thm}
\label{thm:breuer-major2}
Let $\textbf{y}_\textnormal{\tiny{\textbullet}} \coloneqq (\textbf{y}_t)_{t\in \mathbb{N}} $ be a $\mathbb{R}^\ell$-valued centered stationary Gaussian process. Let $f_1,f_2,\ldots,f_{d} : \mathbb{R}^\ell \rightarrow \mathbb{R} $ be measurable functions that have at least Hermite rank $q \in \mathbb{N}$, with respect to $\textbf{y}_\textnormal{\tiny{\textbullet}}$, and let $f_j \circ \textbf{y}_\textnormal{\tiny{\textbullet}}$ be square-integrable for every $j\in\{1,\ldots d\}$. Additionally, let the series of covariances satisfy
\begin{align*}
\sum_{\tau=0}^\infty \left|\ddot{\textnormal{S}}_\tau\left[y_1^{(j)}, y_1^{(k)} \right]\right|^q = \sum_{\tau=0}^\infty \left|\mathbb{E}\left[y_1^{(j)} y_{1+\tau}^{(k)} \right]\right|^q < +\infty, \quad \forall j,k \in \left\{1,\ldots,\ell\right\}.
\end{align*}
Then, the $\mathbb{R}^d$-sequence,
\begin{align*}
\frac{1}{\sqrt{T}}\begin{pmatrix}\sum_{t=1}^T \left(f_1(\textbf{y}_t) - \mathbb{E}\left[f_1(\textbf{y}_t)  \right]  \right) &\cdots &\sum_{t=1}^T \left(f_d(\textbf{y}_t) - \mathbb{E}\left[f_d(\textbf{y}_t)  \right]  \right)  \end{pmatrix}^\top
\end{align*}
converges in distribution, as $T\rightarrow \infty$, to a centered Gaussian vector $\boldsymbol{\rho}^\top = \begin{pmatrix}\rho_1 &\cdots &\rho_d\end{pmatrix}$ with finite covariances $\mathbb{E}\left[\rho_j \rho_k \right] $ equal to,
\begin{align*}
{\textnormal{S}}_0\left[f_j(\textbf{y}_1), f_k(\textbf{y}_1) \right]+ \sum_{\tau=1}^\infty \left( \ddot{\textnormal{S}}_\tau \left[f_j(\textbf{y}_1), f_k(\textbf{y}_1) \right]  + \ddot{\textnormal{S}}_\tau \left[f_k(\textbf{y}_1),f_j(\textbf{y}_1) \right]  \right).
\end{align*}
\end{thm}
The proof of Theorem \ref{thm:breuer-major2} is omitted here. Theorem \ref{thm:breuer-major2} follows directly from the univariate Breuer-Major Theorem, given in \cite{breuer-major}, and by using the usual Cram\'{e}r-Wold device, see e.g. \cite{billingsley2012}[Theorem 29.4]. A similar version of Theorem \ref{thm:breuer-major2} can also be found in Section 5 of \cite{arcones1994limit}.

We next present the following assumption which enables us to find the asymptotic behaviour of the unmixing matrix estimator using the Breuer-Major Theorem.

\begin{ass}
\label{ass:summablecov}
Let $\textbf{x}_\textnormal{\tiny{\textbullet}} \coloneqq \left(\textbf{x}_t\right)_{t\in \mathbb{N}}$ be a process that satisfies Definition \ref{model:BSS} with the trivial mixing matrix $\textbf{A} = \textbf{I}_d$ and let $\textbf{X}$ be a  $ \mathbb{C}^{T \times d}$-valued, $1\leq d < T < \infty$, sampled stochastic process generated by $\textbf{x}_\textnormal{\tiny{\textbullet}}$. Denote  $\textbf{z}_\textnormal{\tiny{\textbullet}} = (\textbf{b}_t +i\textbf{c}_t)_{t\in \mathbb{N}} =\textbf{b}_\textnormal{\tiny{\textbullet}}  + i\textbf{c}_\textnormal{\tiny{\textbullet}}  $. We assume that there exists $\ell\in \mathbb{N}$ and a centered $\mathbb{R}^{\ell}$-valued stationary Gaussian process
$\boldsymbol{\eta}_{\textnormal{\tiny{\textbullet}}}$, such that, for all $k\in\{1,\ldots d\}$, the component  $\textbf{b}_\textnormal{\tiny{\textbullet}}^{(k)}$ has the same finite-dimensional distributions and, asymptotically, the same autocovariance function as $\tilde{f}_k(\boldsymbol{\eta}_{\textnormal{\tiny{\textbullet}}})$, for some function $\tilde{f}_k : \mathbb{R}^{\ell} \rightarrow \mathbb{R}$. Similarly, we assume that, for all $k\in\{1,\ldots d\}$,  the component  $\textbf{c}_\textnormal{\tiny{\textbullet}}^{(k)}$ has the same finite-dimensional distributions and, asymptotically, the same autocovariance function as $\tilde{f}_{k+d}(\boldsymbol{\eta}_{\textnormal{\tiny{\textbullet}}})$, for some function $\tilde{f}_k : \mathbb{R}^{\ell} \rightarrow \mathbb{R}$. Furthermore, we assume that
$\mathbb{E}[|\tilde{f}_k(\boldsymbol{\eta}_1)|^4] < +\infty$, $\forall k \in \{1,\ldots,2d\}$,  and that $r_{\boldsymbol{\eta}}^{(j,k)}(t) = \mathbb{E}\left[ \boldsymbol{\eta}^{(j)}_{t+1} \boldsymbol{\eta}^{(k)}_1\right]$ satisfies $r_{\boldsymbol{\eta}}^{(j,j)}(0)=1$, $\forall  j\in \{1,\ldots,\ell\} $, and
\begin{align*}
 \sum_{t=1}^\infty \left|r_{\boldsymbol{\eta}}^{(j,k)}(t) \right| < + \infty , \quad \forall j,k\in \left\{1,\ldots,{\ell}\right\}.
\end{align*}
\end{ass}
Note that, the finite dimensional distributions of the $\mathbb{R}$-valued stochastic process ${c}_\textnormal{\tiny{\textbullet}}$ is the collection of probability measures defined as,
\begin{align*}
\left\{ \mathbb{P}\left[ \left\{{c}_{t_1}\!\! \in \! B_1,\ldots, {c}_{t_m}\!\!  \in \! B_m \right\} \right] \mid  m \! \in \! \mathbb{N} :  \forall B_j \in \mathcal{B}(\mathbb{R}) \textnormal{ and } \{t_1, \ldots, t_m\}\! \subset \mathbb{N} \right\}.
\end{align*}

We  again want to emphasize that a wide class of stochastic processes satisfy Assumption \ref{ass:summablecov}. Indeed, we allow an
arbitrary dimension $\ell$ for the driving Gaussian process $\boldsymbol{\eta}_{\textnormal{\tiny{\textbullet}}}$ and  arbitrary (summable) covariance structures. In comparison, it was shown in \cite{viitasaari2017} that in many cases it would suffice to relate only one stationary Gaussian process $\boldsymbol{\eta}^{(j)}_{\textnormal{\tiny{\textbullet}}}$ to each function~$\tilde{f}_j$.

We are now ready to consider the asymptotic distribution for the unmixing matrix estimator under Assumption \ref{ass:summablecov}.

\begin{thm}
\label{thm:covconvergence}

Let Assumption \ref{ass:summablecov} be satisfied and let $\hat{g} : \textbf{C} \mapsto (\textbf{C} - \textbf{1}_T \hat{\boldsymbol{\mu}}^\top)\hat{\boldsymbol{\Gamma}}^\top $ be a $T$-indexed sequence of finite sample solutions. Then, under the trivial mixing scenario $\textbf{A} = \textbf{I}_d$, there exists a $T$-indexed sequence $\xhatt{\textbf{J}}$ of phase-shift matrices, such that
\begin{align*}
\sqrt{T} \cdot \textnormal{vec}\left(\xhatt{\textbf{J}}\hat{\boldsymbol{\Gamma}} - \textbf{I}_d\right) \xrightarrow[T\rightarrow \infty]{\mathcal{D}} \boldsymbol{\nu} \sim \mathcal{N}_{d^2}(\textbf{0}, \boldsymbol{\Sigma}_{\boldsymbol{\nu}}, \textbf{P}_{\boldsymbol{\nu}}).
\end{align*}
The exact forms for $\boldsymbol{\Sigma}_{\boldsymbol{\nu}}$ and $\textbf{P}_{\boldsymbol{\nu}}$ are given in Appendix \ref{app:B}.
\end{thm}
Note that, it is possible to present Theorem \ref{thm:covconvergence} with even weaker assumptions. In the current formulation, we require that the Gaussian process $\textbf{y}_\textnormal{\tiny{\textbullet}}$ (equivalently, $\boldsymbol{\eta}_{\textnormal{\tiny{\textbullet}}}$) has summable covariances and cross-covariances.  By studying the exact Hermite ranks of the underlying functions, weaker summability conditions would suffice (cf. Theorem \ref{thm:breuer-major2}). However, given the Hermite ranks of the functions $f_j$, it is in general impossible to say anything about the Hermite ranks of transformations such as $f_j^2$ arising from $\hat{\textbf{S}}_0$, see the proof of Theorem~\ref{thm:covconvergence}. On the other hand, assuming Hermite rank equal to one is a very natural assumption in many occasions. For example, the modeling approach given in \cite{viitasaari2017} sets the rank of $f_j$ to equal 1. Moreover, Hermite rank 1 is stable in a sense that even small perturbations in a function with higher Hermite rank leads to rank 1 again. For further discussion on the stability of Hermite rank 1, we refer to \cite{bai2014generalized}.

\subsection{Notes on non-central limit theorems}
\label{sec:noncentral}
In this section, we provide examples where the convergence rate of the unmixing estimator differs from the standard $\sqrt{T}$ and where the limiting distribution is non-Gaussian. Such situations arise, especially, when the convergence summability condition of Theorem \ref{thm:breuer-major2} does not hold.

\begin{ass}
\label{ass:longrange-model}
Let $\textbf{x}_\textnormal{\tiny{\textbullet}} \coloneqq \left(\textbf{x}_t\right)_{t\in \mathbb{N}}$ be a process that satisfies Definition~\ref{model:BSS} with the trivial mixing matrix $\textbf{A} = \textbf{I}_d$ and let $\textbf{X}$ be a $\mathbb{C}^{T \times d}$-valued, $1\leq d < T < +\infty$, sampled stochastic process generated by $\textbf{x}_\textnormal{\tiny{\textbullet}}$ and $\textbf{z}_\textnormal{\tiny{\textbullet}}=  \textbf{b}_\textnormal{\tiny{\textbullet}}+ i \textbf{c}_\textnormal{\tiny{\textbullet}}$.  We assume that there exists centered $\mathbb{R}$-valued stationary Gaussian processes
$\boldsymbol{\eta}_{\textnormal{\tiny{\textbullet}}}^{(k)}$, $k\in\{1,2\ldots,2d\}$, such that, for all $k\in\{1,\ldots d\}$, the component  $\textbf{b}_\textnormal{\tiny{\textbullet}}^{(k)}$ has the same finite-dimensional distributions and, asymptotically, the same autocovariance function as $\tilde{f}_k(\boldsymbol{\eta}^{(k)}_{\textnormal{\tiny{\textbullet}}})$, for some function $\tilde{f}_k : \mathbb{R} \rightarrow \mathbb{R}$. Similarly, we assume that, for all $k\in\{1,\ldots d\}$,  the component  $\textbf{c}_\textnormal{\tiny{\textbullet}}^{(k)}$ has the same finite-dimensional distributions and, asymptotically, the same autocovariance function as $\tilde{f}_{k+d}(\boldsymbol{\eta}^{(k+d)}_{\textnormal{\tiny{\textbullet}}})$, for some function $\tilde{f}_{k+d} : \mathbb{R} \rightarrow \mathbb{R}$. Furthermore, we assume that
$\mathbb{E}[|\tilde{f}_k(\boldsymbol{\eta}^{(k)}_1)|^4] < +\infty$, $\forall k \in \{1,\ldots,2d\}$, and that the Gaussian processes $\boldsymbol{\eta}_{\textnormal{\tiny{\textbullet}}}^{(1)},\boldsymbol{\eta}_{\textnormal{\tiny{\textbullet}}}^{(2)},\ldots,\boldsymbol{\eta}_{\textnormal{\tiny{\textbullet}}}^{(2d)}$ are mutually independent.
\end{ass}

In comparison to Assumption \ref{ass:summablecov}, we here assume that real- and imaginary parts of each component is driven by a single Gaussian process. As discussed in Section \ref{sec:bmajor}, this is not a huge restriction. In addition, we assume independent components and that the real- and imaginary parts of each component are independent. Although, independence is a common assumption in the blind source separation literature, we note that our results can be extended to cover dependencies as long as, for every $j,k\in\{1,2,\ldots,2d\}$, the cross-covariances between $\boldsymbol{\eta}^{(k)}$ and $\boldsymbol{\eta}^{(j)}$ are negligible compared to the decay of the autocovariance functions $r_{\boldsymbol{\eta}^{(j)}}(t)$. 

If the autocovariance functions $r_{\boldsymbol{\eta}^{(j)}}$ satisfy,
\begin{equation}
\label{eq:short-range}
\sum_{t=1}^\infty |r_{\boldsymbol{\eta}^{(j)}}(t)| < +\infty, \quad \forall j\in\{1,\ldots 2d\},
\end{equation}
then we are in the situation of Theorem \ref{thm:covconvergence}. More generally, a weakly stationary series is called short-range dependent if the corresponding autocovariance function $r$ satisfies Equation \eqref{eq:short-range}. Hence, we assume that at least one of the functions $r_{\boldsymbol{\eta}^{(j)}}$ is not summable. For such components, we assume the following long-range dependence condition.
\begin{dfn}
We say that a $\mathbb{R}$-valued weakly stationary process ${\eta}_{\textnormal{\tiny{\textbullet}}}$ is long-range dependent if the autocovariance function $r_{{\eta}}$ satisfies
\begin{equation}
\label{eq:long-range}
\lim_{k\to \infty} k^{2-2H}r_{{\eta}}(k) =C
\end{equation}
for some $H \in \left[1/2,1\right)$ and $C\in (0,\infty)$.
\end{dfn}
Note that, the definition of long-range dependence varies in the literature. For details on long-range dependent processes and their different definitions, we refer to \cite{beran,beran-et-al,sam}. 

There is a large literature on limit theorems under long-range dependence. See, e.g., \cite{Dobrushin-Major-1979,Taqqu-1975,bai-taq2} for a few central works on the topic. In the case of long-range dependent processes, the rate of convergence of the normalized mean is slower than the usual $\sqrt{T}$ and the limiting distribution depends on the corresponding Hermite rank. More precisely, the limiting distribution follows a so-called $q$-Hermite distribution, where $q$ is the Hermite rank of the underlying function. Note that, $q$-Hermite distributions can be fully characterized with the corresponding Hermite rank $q$ and a self-similarity parameter $H$, i.e., the Hurst index. In the stable case $q=1$, we obtain a Gaussian limit, and in the case $q=2$ we obtain the so-called \emph{Rosenblatt distribution} that is not Gaussian. Similarly, for $q\geq 3$ the limiting distribution is not Gaussian. For details on Hermite distributions and processes, we refer to \cite{Embrechts-Maejima-2002,tudor, bai-taq2}. In particular, we apply the following known result (see \cite[Theorem 1]{Dobrushin-Major-1979}).
\begin{pro}
\label{pro:longrange-limit-general}
Let ${\eta}_{\textnormal{\tiny{\textbullet}}}$ be a $\mathbb{R}$-valued stationary Gaussian process with autocovariance function $r_{\eta}$ such that $r_{\eta}(0)=1$ and $r_{\eta}$ satisfies Equation \eqref{eq:long-range} for some $H\in\left[1/2,1\right)$. Let $f$ be a function such that $\mathbb{E}[(f({\eta}_1))^2] < +\infty$ and the Hermite rank of $f$ equals $q$. If
$
q(2H-2)>-1,
$
then, for some constant $C = C_{f,\eta} >0$, we have
$$
T^{q(1-H)-1} \sum_{t=1}^T \left(f({\eta}_t) - \mathbb{E}[f({\eta}_t)]\right) \xrightarrow[T\rightarrow \infty]{\mathcal{D}}  C Z_q,
$$
where $Z_q$ follows a $q$-Hermite distribution.
\end{pro}
In view of Proposition \ref{pro:longrange-limit-general}, in the long-range dependent case the Hermite rank plays a crucial role. Thus, one cannot pose general results without any a priori knowledge on the ranks. Indeed,
it can be shown --- see, e.g., \cite[Equation 4.26]{beran-et-al} with $d=H-1/2$ --- that if $f$ has Hermite rank $q$, and $r$ satisfies Equation \eqref{eq:long-range} with some $H$ such that $q(2H-2)>-1$, then
\begin{equation}
\label{eq:long-range-variance}
\lim_{T\to \infty}T^{q(2-2H)}\mathbb{E}\left[\left(\frac{1}{T}\sum_{t=1}^T \left(f(\boldsymbol{\eta}_t) - \mathbb{E}[f(\boldsymbol{\eta}_t)]\right)\right)^2\right] =C.
\end{equation}
This justifies the normalization in Proposition \ref{pro:longrange-limit-general} and also gives the constant explicitly, as $\mathbb{E}[Z_q] = 0$ and $\mathbb{E}[(Z_q)^2] = 1$. We pose the following assumption for the autocovariances $r_{\boldsymbol{\eta}^{(j)}}$ and Hermite ranks.
\begin{ass}
\label{ass:longrange-ranks}
We assume that for every $j \in \{1,2,\ldots,2d\}$ the autocovariance functions $r_{\boldsymbol{\eta}^{(j)}}$ with $r_{\boldsymbol{\eta}^{(j)}}(0)=1$   satisfy either Equation \eqref{eq:short-range} or Equation \eqref{eq:long-range}, and that Equation \eqref{eq:long-range} is satisfied for at least one index. Let $I\subset \{1,2,\ldots,2d\}$ denote the set of indices for which Equation \eqref{eq:long-range} holds, and for every $i\in I$, denote the Hermite rank of $x\mapsto \tilde{f}_i(x)$ by $q_{1,i}$ and the Hermite rank of $x\mapsto \left[\tilde{f}_i(x)-\mathbb{E}(\tilde{f}_i(\boldsymbol{\eta}^{(i)}))\right]^2$ by $q_{2,i}$. For any indices $j,k\in I$, $j\neq k$, we assume that 
\begin{equation}
\label{eq:maximum}
\max_{i\in I} q_{2,i}(2H_{i}-2) > \max_{j,k\in I}\left\{ q_{1,k}(2H_{k}-2)+ q_{1,j}(2H_{j}-2), -1\right\},
\end{equation}
and that for any $k \in I$
\begin{equation}
\label{eq:maximum2}
\max_{i\in I} q_{2,i}(2H_{i}-2) \geq q_{1,k}(4H_{k}-4).
\end{equation}
\end{ass}
We stress that Assumption \ref{ass:longrange-ranks} is not very restrictive. We allow a combination of short- and long-range processes $\boldsymbol{\eta}^{(j)}_{\textnormal{\tiny{\textbullet}}}$ and we assume that at least one of the processes is long-range dependent, since otherwise we may apply Theorem \ref{thm:covconvergence}. In this respect, Condition \eqref{eq:maximum} guarantees that at least one of the $[\tilde{f}_i(x)-\mathbb{E}(\tilde{f}_i(\boldsymbol{\eta}^{(i)}))]^2$  is long-range dependent. Indeed, if $\max_{i\in I}q_i(2H_i-2) < -1$, then Theorem \ref{thm:covconvergence} applies again. Moreover,  one could also study the limiting case $\max_{i\in I}q_i(2H_i-2) = -1$. Then, one usually obtains a Gaussian limit, but with a rate given by $\sqrt{{T}/{\log(T)}}$. Our more restrictive Assumptions \eqref{eq:maximum} and \eqref{eq:maximum2} are posed in order to apply known results on limit theorems in the long-range dependent case. Indeed, if these conditions were violated, then we would need to study the convergence of complicated dependent random vectors towards some combinations of Hermite distributions. Unfortunately, to the best of our knowledge, there are no studies in this topic that would fit into our general setting (on some results in special cases, see e.g. \cite{bai-taq2}). 

In view of Proposition \ref{pro:longrange-limit-general}, the rate of convergence, $T^{\gamma}$, of our unmixing estimator arises from Equation \eqref{eq:maximum}. We set 
\begin{equation}
\label{eq:rate-gamma}
\gamma := -\frac12 \max_{i\in I}q_{2,i}(2H_i-2).
\end{equation}
By Equation \eqref{eq:maximum}, we have that $\gamma<\frac12$, meaning that our rate is slower than the usual $\sqrt{T}$.

The following results show that cross-terms do not contribute to the limit.
\begin{pro}
\label{pro:cross-disappear}
Let Assumptions \ref{ass:longrange-model} and \ref{ass:longrange-ranks} be satisfied. Then for every $j\neq k$, we have that,
\begin{equation}
\label{eq:cross-disappear1}
T^{\gamma-1}\sum_{t=1}^T \left( \tilde{f}_k(\boldsymbol{\eta}^{(k)}_t)-\mathbb{E}[\tilde{f}_k(\boldsymbol{\eta}^{(k)}_t)] \right)\left( \tilde{f}_j(\boldsymbol{\eta}^{(j)}_t)-\mathbb{E}[\tilde{f}_j(\boldsymbol{\eta}^{(j)}_t)] \right)\xrightarrow[T\rightarrow \infty]{\mathcal{L}_2} 0
\end{equation}
and
\begin{equation}
\label{eq:cross-disappear2}
T^{\gamma-1}\sum_{t=1}^T \left( \tilde{f}_k(\boldsymbol{\eta}^{(k)}_t)-\mathbb{E}[\tilde{f}_k(\boldsymbol{\eta}^{(k)}_t)] \right)\left( \tilde{f}_j(\boldsymbol{\eta}^{(j)}_{t+\tau})-\mathbb{E}[\tilde{f}_j(\boldsymbol{\eta}^{(j)}_{t})] \right)\xrightarrow[T\rightarrow \infty]{\mathcal{L}_2} 0
\end{equation}
where $  \xrightarrow[ ]{\mathcal{L}_2}$ denotes convergence in the space $\mathcal{L}_2(\Omega,\mathbb{P})$.  Consequently, the convergence also holds in probability.
\end{pro}
In the presence of long-range dependency, the mean estimator  $\hat{\boldsymbol{\mu}}[\textbf{X}]$ in the definition of $\tilde{\textbf{S}}_{\tau} \left[ \textbf{X} \right]$ has to be studied separately as it may contribute to the limit. 
\begin{lma}
\label{lma:mu-conv}
Let Assumptions \ref{ass:longrange-model} and \ref{ass:longrange-ranks} be satisfied. Let $\hat{\boldsymbol{\mu}} \coloneqq \hat{\boldsymbol{\mu}}[\textbf{X}]$ and $\tilde{\boldsymbol{\mu}} =  \hat{\boldsymbol{\mu}}-\boldsymbol{\mu}$.
Then for every element $(j,k)$, such that $j,k\in\{1,2,\ldots,d\}$ and $j\neq k$, we have
$
[T^{\gamma}\tilde{\boldsymbol{\mu}}\tilde{\boldsymbol{\mu}}^{\textnormal{H}}]_{jk}  \xrightarrow[T\rightarrow \infty]{\mathcal{L}_1} 0,
$
where $  \xrightarrow[ ]{\mathcal{L}_1}$ denotes convergence in the space $\mathcal{L}_1(\Omega,\mathbb{P})$. Consequently, the convergence also holds in probability.
\end{lma}
In order to obtain the limiting distribution for
$T^\gamma(\xhatt{\textbf{J}}\hat{\boldsymbol{\Gamma}} - \textbf{I}_d)$, we introduce some notation. For each $k\in\{1,2,\ldots,2d\}$ we define Hermite processes $Z_{q_{1,k}}$ and $Z_{q_{2,k}}$, such that for a given $k$, we allow $Z_{q_{1,k}}$ and $Z_{q_{2,k}}$ to be dependent on each other, while both are independent of $Z_{q_{1,j}}, Z_{q_{2,j}}$ for $j\neq k$. We also introduce constants $C_{1,k},C_{2,k}$ in the following way: if $k \in I$ is such that equality holds in \eqref{eq:maximum2}, i.e., $\max_{i\in I}q_{2,i}(2H_i-2) = 2[q_{1,k}(2H_k-2)]$, then we let $C_{1,k}$ to be the constant given in Proposition \ref{pro:longrange-limit-general} associated to the limit
$$
T^{q_{1,k}(1-H_k)-1} \sum_{t=1}^T \left(\tilde{f}_k(\boldsymbol{\eta}^{(k)}_t) - \mathbb{E}[\tilde{f}_k(\boldsymbol{\eta}^{(k)}_t)]\right) \xrightarrow[T\rightarrow \infty]{\mathcal{D}}  C_{1,k}Z_{q_{1,k}}.
$$
Otherwise we set $C_{1,k}=0$. Similarly, if $k\in I$ is such that
\begin{equation}
\label{eq:needed-for-contribute}
q_{2,k}(2H_k-2) = \max_{i\in I}q_{2,i}(2H_i-2),
\end{equation}
we let $C_{2,k}$ to be the constant given in Proposition \ref{pro:longrange-limit-general} associated to the limit
$$
T^{q_{2,k}(1-H_k)-1}\sum_{t=1}^T \left( h_{k}(\boldsymbol{\eta}^{(k)}_t) - \mathbb{E}\left[ h_{k}(\boldsymbol{\eta}^{(k)}_t)  \right]\right) \xrightarrow[T\rightarrow \infty]{\mathcal{D}}  C_{2,k}Z_{q_{2,k}},
$$
where
$$
h_{k}(\boldsymbol{\eta}^{(k)}_t) =   \left(\tilde{f}_k(\boldsymbol{\eta}^{(k)}_t) - \mathbb{E}[\tilde{f}_k(\boldsymbol{\eta}^{(k)}_t)]\right)^2.
$$
Otherwise we set $C_{2,k} = 0$. The following theorem is our main result of this subsection.
\begin{thm}
\label{thm:limit-longrange}
Let Assumptions \ref{ass:longrange-model} and \ref{ass:longrange-ranks} hold. Then,
\begin{align*}
T^{\gamma} \left(\xhatt{\textbf{J}}\hat{\boldsymbol{\Gamma}} - \textbf{I}_d\right) \xrightarrow[T\rightarrow \infty]{\mathcal{D}} \frac{1}{2}\boldsymbol{\Upsilon},
\end{align*}
where $\boldsymbol{\Upsilon}$ is a $  \mathbb{R}^{d\times d}$-valued diagonal matrix with elements given by 
$$
(\boldsymbol{\Upsilon})_{jj} \sim C_{2,j}Z_{q_{2,j}} + C_{1,j}^2Z^2_{q_{1,j}} + C_{2,j+d}Z_{q_{2,j+d}} + C_{1,j+d}^2Z^2_{q_{1,j+d}}.
$$
\end{thm}
We remark that it might be that most of the elements vanish. Indeed, the coefficient $C_{2,j}$ ($C_{2,j+d}$, respectively) is non-zero only if the real part (imaginary part, respectively) of $\left[f_j(x)-\mathbb{E}(f_j(\boldsymbol{\eta}))\right]^2$ converges with the correct rate $T^{\gamma}$. In this case, $C_{1,j}$ ($C_{1,j+d}$, respectively) is non-zero only if the corresponding element in $\tilde{\boldsymbol{\mu}}$ converges with the rate $T^{\gamma/2}$.

\section{Image source separation}
\label{sec:image}

In this section, we present an image source separation example, where complex-valued signals are utilized. The restoration of images, which have been mixed together by some transformation, is a classical example in blind source separation (BSS), see e.g. \cite{nordhausen2008tools}. In the literature, the images source separation is usually only considered for black and white  images. In the black and white case, we can choose some interval, for example [0,1], and assign every pixel a value belonging to the interval, e.g., such that the pixels with value 1 are black and pixels with value 0 are white. Different shades of gray can then be generated by considering values between 0 and 1. Since there is only a single color parameter associated with a single pixel, the black and white image source problem can be straightforwardly formulated as a real-valued BSS problem.

In our example, we consider colored photographs and conduct the example using the statistical software R (\cite{Rbase}). We start with three equally sized photographs containing $960\times 720$ pixels and import them to R using the package JPEG, \cite{Rjpeg}. The imported images are stored in a $960\times 720 \times 3$ array format, such that a red-green-blue color parameter triplet, denoted as $(R,G,B)$, is assigned to every pixel. The $R,G$ and $B$ parameters are discrete, such that they can take 256 distinct values between zero and one.

\begin{figure}
  \centering
    \includegraphics[width=1\textwidth]{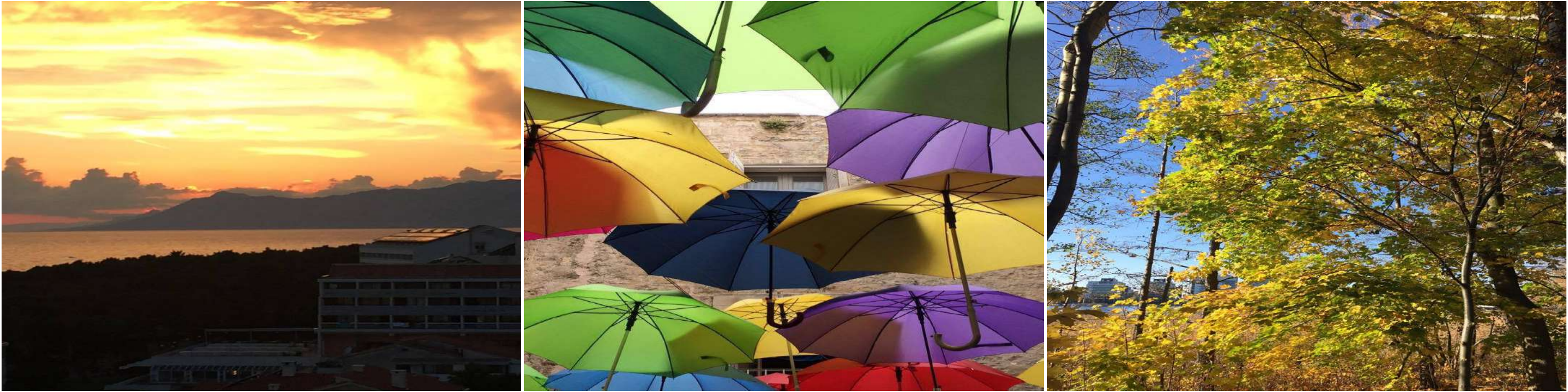}
  \caption{Original color corrected images.}
\label{fig:original}
\end{figure}

Note that the color parameters $R,G$ and $B$ form a discrete color cube. As a preliminary step, we transform the original images such that the color parameter triplet of every pixel is projected to the closest point on the surface of the color cube. Hereby, we have for every pixel that either $\min(R,G,B) = 0$ or $\max(R,G,B) = 1$, without excluding the possibility of them both being simultaneously true. The color corrected images are presented in Figure \ref{fig:original}. The color correction enables us to define a bijective transformation between the color cube surface, which we require for plotting the images, and the complex plane, where we perform the mixing and unmixing procedures.  

After the color correction, we apply a bijective cube to unit sphere transformation for every color triplet. Then, we use the well-known stereographic projection, which is an almost bijective transformation between the unit sphere and the complex plane. The stereographic projection is bijective everywhere except for the north pole of the unit sphere. For almost all photographs, we can choose the color coordinate system such that no pixel has a color triplet located on the north pole. This holds for our example and we can apply the inverse mappings in order to get back to the color cube surface.

We then have a single complex-number corresponding to every pixel of the color corrected images and we can present the images in a $\mathbb{C}^{960\times 720}$-matrix form. We denote the images with complex-valued elements as $\textbf{Z}_1, \textbf{Z}_2$ and $\textbf{Z}_3$ and apply the following transformation,
\begin{align*}
\begin{pmatrix} \textnormal{vec}(\textbf{X}_1)^\top \\ \textnormal{vec}(\textbf{X}_2)^\top \\ \textnormal{vec}(\textbf{X}_3)^\top \end{pmatrix}^\top = \begin{pmatrix} \textnormal{vec}(\textbf{Z}_1)^\top \\ \textnormal{vec}(\textbf{Z}_2)^\top \\ \textnormal{vec}(\textbf{Z}_3)^\top \end{pmatrix}^\top\left( \textbf{I}_3 + \textbf{E} + i\textbf{F}\right),
\end{align*}
 where $\textbf{I}_3$ is the identity matrix and the elements of $\textbf{E}$ and $\textbf{F}$ were generated independently from the univariate uniform distribution $\texttt{Unif(-1/2, 1/2)}$. The complex-valued mixed images $\textbf{X}_1,\textbf{X}_2$ and $\textbf{X}_3$ can then be plotted by performing the chain of  inverse mappings in reverse order. The corresponding images are presented in Figure \ref{fig:mixed}.

\begin{figure}
  \centering
    \includegraphics[width=1\textwidth]{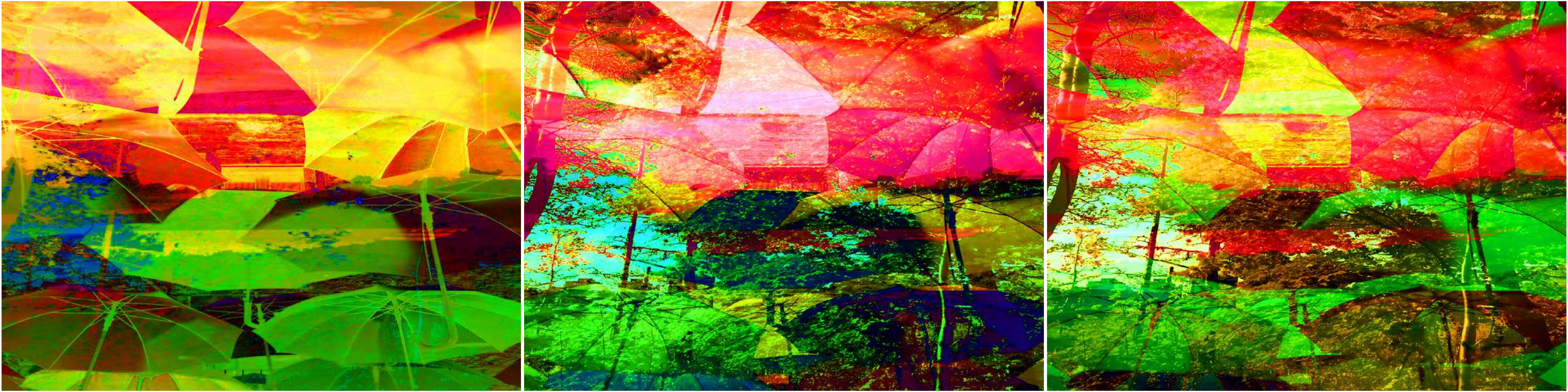}
  \caption{Mixed images. }
\label{fig:mixed}
\end{figure}

We then apply the unmixing procedure presented in Section \ref{sec:asympBSS}, with several different lag parameters $\tau$. In controlled settings, where the true mixing matrix is known, the minimum distance index (MD) index is a straightforward way to compare the performance of different estimators, see \cite{lietzen2016minimum,lietzen2019} for the complex-valued formulation. The MD index is a performance index, defined between zero and one, where zero corresponds to the separation being perfect. In addition, the MD index ensures that comparison between different estimators is fair, this is especially important in this paper since under our model assumptions, different estimators do not necessarily estimate the same population quantities. For additional details, see \cite{ilmonen2012asymptotics,lietzen2016minimum,lietzen2019}

We tried several different lag parameters and the best performance, in the light of the MD index and visual inspection, was attained with $\tau =1$ and the corresponding MD index value was approximately $0.173$. The unmixed images obtained using the unmixing procedure with $\tau=1$ are presented in Figure \ref{fig:unmixed}. In addition, for example, lag parameter choices $\tau = 2,3,961,962,963$ provided only slightly worse results, the corresponding MD index values were between 0.177 and 0.183. Note that the autocovariances are calculated from the vectorized images, and for example the first and the 961st entries are neighboring pixels in the unvectorized images.

\begin{figure}
  \centering
    \includegraphics[width=1\textwidth]{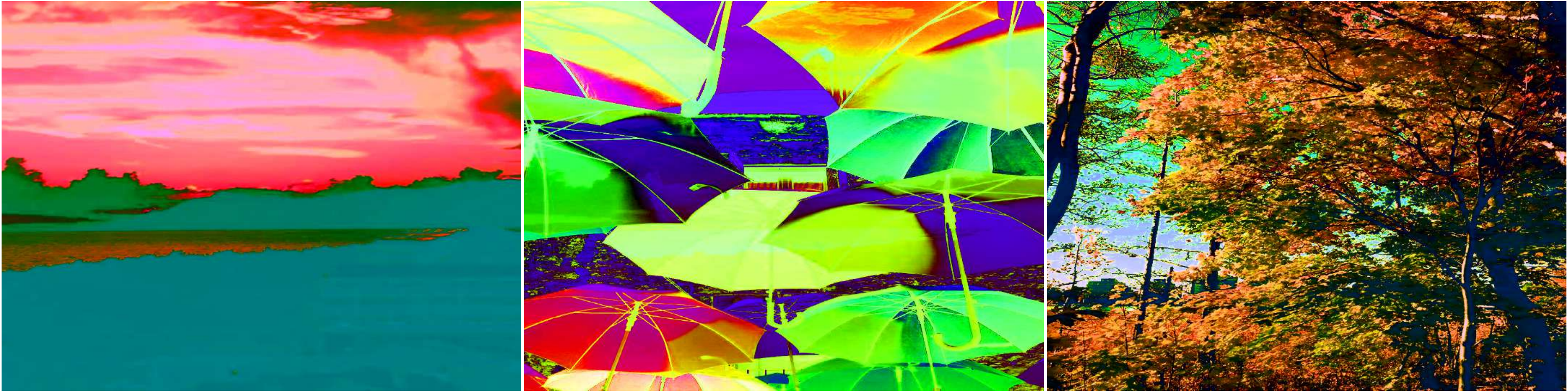}
  \caption{Unmixed images using the AMUSE procedure with $\tau =1$.}
\label{fig:unmixed}
\end{figure}

For comparison, we also applied complex-valued versions of the fourth-order blind identification (FOBI) and the second-order blind identification (SOBI) procedures, see \cite{ilmonen2013,belouchrani1997blind}. In this example, the the unmixing  procedure with $\tau=1$ outperformed both FOBI and SOBI. The MD index of the FOBI procedure was approximately 0.318. We applied the SOBI procedure with several different lag parameter combinations, and the best performance, when excluding the cases that correspond to the AMUSE transformation, was obtained with $\boldsymbol{\tau} = \{1,962\}$ and the corresponding MD index value was approximately 0.177.

When comparing the original color corrected images in Figure \ref{fig:original} and the unmixed images in Figure \ref{fig:unmixed}, the shapes seem to match almost perfectly. However, the difference of the colors is significant. The colors vary since the complex phase is not uniquely fixed in our model, recall Lemma \ref{lma:permphase}. In addition, recall that under our model, solutions that are a phase-shift away from each other are considered to be equivalent. In Figure \ref{fig:rotated}, we present the first unmixed image, produced by the unmixing procedure with $\tau=1$, under three equivalent solutions. The images in Figure \ref{fig:rotated} are obtained such that we first find an unmixing estimate $\hat{\boldsymbol{\Gamma}}$ and then right-multiply the obtained estimate with a diagonal matrix with diagonal entries $\exp(i\theta)$. In Figure \ref{fig:rotated}, the leftmost image is obtained with $\theta = \pi/4$, the middle is obtained with $\theta = 3\pi/4$ and the rightmost image is obtained with $\theta=5\pi/4$.

\begin{figure}
  \centering
    \includegraphics[width=1\textwidth]{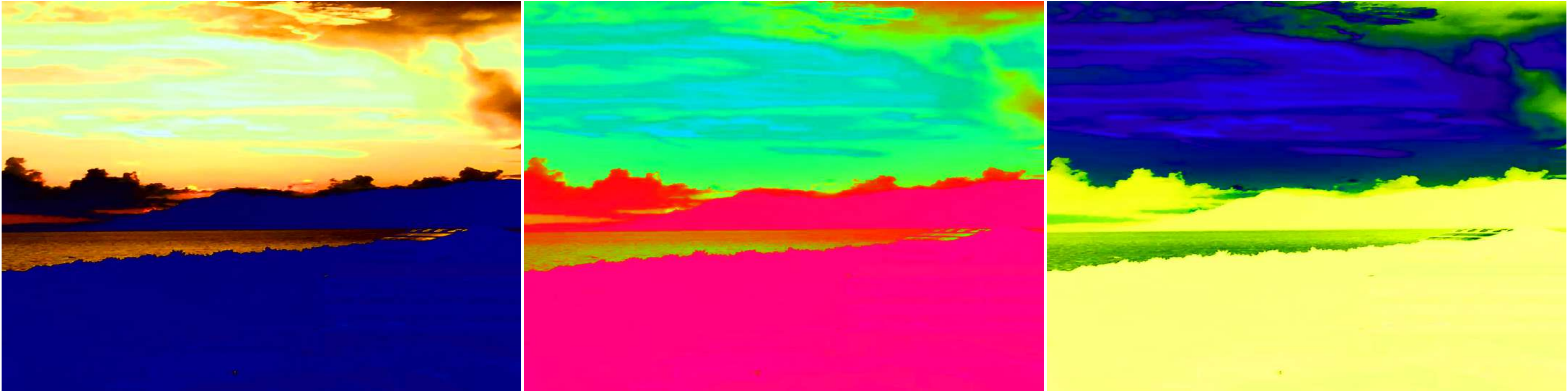}
  \caption{Unmixed images produced by equivalent solutions.}
\label{fig:rotated}
\end{figure}

\section{Acknowledgements}
N. Lietz\'{e}n gratefully acknowledges financial support from the Emil Aaltonen Foundation (grant 180144 N). The authors would like to thank Katariina Kilpinen for providing the photographs for Section \ref{sec:image}.

\appendix
\addcontentsline{toc}{section}{Appendices}
\renewcommand{\thesubsection}{\Alph{subsection}}

\section{Appendix}
\label{app:B}
Appendix \ref{app:B} contains the proofs of the results given in this paper. We begin by presenting a lemma, that we later utilize in the upcoming proofs.
\begin{lma}
\label{lma:inverse}
Let $\hat{\boldsymbol{\Sigma}}$ be a $T$-indexed sequence of $\mathbb{C}^{d\times d}$-valued nonsingular estimates and let $\alpha_T$ be a real-valued sequence that satisfies $\alpha_T \uparrow \infty$ as $T\rightarrow \infty$. Furthermore, let  $\alpha_T(\hat{\boldsymbol{\Sigma}} - \textbf{I}_d)=\mathcal{O}_p(1)$. Then, the following two statements hold.
\begin{itemize}
\item[(i)] $\hat{\boldsymbol{\Sigma}}^a\xrightarrow[T\rightarrow\infty]{\mathbb{P}} \textbf{I}_d$, where $a \in\left\{-1,-\frac{1}{2},1\right\}.$

\item[(ii)] $\alpha_T(\hat{\boldsymbol{\Sigma}}^a - \textbf{I}_d ) = \mathcal{O}_p(1)$, where $a \in\left\{-1,-\frac{1}{2}\right\}.$

\end{itemize}
\end{lma}

\begin{proof}[Proof of Lemma \ref{lma:inverse}]
The assumption $\alpha_T(\hat{\boldsymbol{\Sigma}} - \textbf{I}_d)=\mathcal{O}_p(1)$, implies that $\hat{\boldsymbol{\Sigma}} - \textbf{I}_d = (1/\alpha_T) \mathcal{O}_p(1) = o_p(1) \mathcal{O}_p(1) = o_p(1)$, that is, $\hat{\boldsymbol{\Sigma}}\xrightarrow[]{\mathbb{P}} \textbf{I}_d$, as $T\rightarrow\infty$.  Note that the matrix inversion and the conjugate symmetric square root of the inversion are continuous transformations in the neighborhood of $\boldsymbol{I}_d.$ Thus, we can apply the continuous mapping theorem, which gives us that $\hat{\boldsymbol{\Sigma}}^a \xrightarrow[]{\mathbb{P}} \boldsymbol{I}_d$, as $T\rightarrow \infty$, for $a\in\{-1,-\frac{1}{2}\}$.

Using part (i) and Slutsky's lemma, we get that the inverse is uniformly tight, since,
\begin{align*}
\alpha_T(\hat{\boldsymbol{\Sigma}}^{-1} - \textbf{I}_d) &= (\hat{\boldsymbol{\Sigma}}^{-1} - \textbf{I}_d) \alpha_T  (\textbf{I}_d - \hat{\boldsymbol{\Sigma}} ) +  \alpha_T  (\textbf{I}_d - \hat{\boldsymbol{\Sigma}} )\\
& = o_p(1) \mathcal{O}_p(1) + \mathcal{O}_p(1)  = \mathcal{O}_p(1).
\end{align*}
For the final part, first note that,
\begin{align*}
(\hat{\boldsymbol{\Sigma}}^{-1} - \textbf{I}_d) = (\hat{\boldsymbol{\Sigma}}^{-\frac{1}{2}}  - \textbf{I}_d) (\hat{\boldsymbol{\Sigma}}^{-\frac{1}{2} }  +\textbf{I}_d).
\end{align*}
Since $\hat{\boldsymbol{\Sigma}}^{-\frac{1}{2} }   + \textbf{I}_d$  converges in probability to $2\textbf{I}_d$ and the matrix inversion is a continuous transformation in the neighborhood of $2\textbf{I}_d$, the  continuous mapping theorem gives us that $(\hat{\boldsymbol{\Sigma}}^{-\frac{1}{2} }  +\textbf{I}_d)^{-1} \xrightarrow[T\rightarrow \infty]{\mathbb{P}} \frac{1}{2}\boldsymbol{I}_d.$ Thus by Slutsky's lemma,
\begin{align*}
\alpha_T(\hat{\boldsymbol{\Sigma}}^{-\frac{1}{2}}  - \textbf{I}_d) =\alpha_T(\hat{\boldsymbol{\Sigma}}^{-1} - \textbf{I}_d)(\hat{\boldsymbol{\Sigma}}^{-\frac{1}{2}} + \textbf{I}_d)^{-1}= \mathcal{O}_p(1).
\end{align*}
\end{proof}

\begin{proof}[Proof of Lemma \ref{lma:rcconvergence}]
Both directions of the claim follow directly from the multivariate version of the continuous mapping theorem. Note that, the mapping $f: (r_1,r_2,\ldots,r_{2d})^\top \mapsto (r_1,r_2,\ldots,r_d)^\top + i (r_{d+1}, r_{d+2},\ldots, r_{2d})^\top : \mathbb{R}^{2d} \rightarrow \mathbb{C}^d$ is homeomorphic, that is, $f$ is continuous and bijective, and the preimage $f^{-1}$ is also continuous.
\end{proof}

\begin{proof}[Proof of Corollary \ref{cor:eqvrealcomplex}]
The corollary follows directly by applying Lemma \ref{lma:rcconvergence} and by calculating the mean, the covariance and the relation matrix of $\textbf{z}$.
\end{proof}

\begin{proof}[Proof of Lemma \ref{lma:permphase}]
Let $\textbf{y}_\textnormal{\tiny{\textbullet}} \coloneqq g \circ \textbf{x}_\textnormal{\tiny{\textbullet}}$. First, let $\textbf{y}_\textnormal{\tiny{\textbullet}}$ be a solution, that is, $\textbf{y}_\textnormal{\tiny{\textbullet}}$  satisfies conditions (1)-(4) of Definition \ref{model:BSS}. Using condition (1), we get that,
\begin{align*}
\mathbb{E}\left[\textbf{y}_t  \right]=\mathbb{E}\left[\boldsymbol{\Gamma}\left(\textbf{x}_t - \boldsymbol{\mu} \right) \right] =\mathbb{E}\left[\boldsymbol{\Gamma}\left(\textbf{A}\textbf{z}_t + \boldsymbol{\mu}_{\textbf{x}} - \boldsymbol{\mu} \right) \right]  = \boldsymbol{\Gamma}\mathbb{E}\left[ \boldsymbol{\mu}_{\textbf{x}} - \boldsymbol{\mu} \right] = \textbf{0},
\end{align*}
and thus $\boldsymbol{\mu}_{\textbf{x}} = \boldsymbol{\mu} $. Next, we can rewrite condition (2) as,
\begin{align*}
\textbf{S}_0[\textbf{y}_t] = \mathbb{E}\left[ \textbf{y}_t \textbf{y}_t^\textnormal{H} \right] = \mathbb{E}\left[ \boldsymbol{\Gamma} \textbf{A} \textbf{z}_t  \left(\boldsymbol{\Gamma} \textbf{A} \textbf{z}_t\right)^\textnormal{H} \right] = \boldsymbol{\Gamma} \textbf{A} \textbf{A}^\textnormal{H} \boldsymbol{\Gamma}^\textnormal{H} = \textbf{I}_d,
\end{align*}
which implies that $ \boldsymbol{\Gamma} \textbf{A}$ is a unitary matrix. Similarly, we can rewrite condition (4) as,
\begin{align*}
\textbf{S}_\tau[\textbf{y}_t] =\frac{1}{2}\left(\mathbb{E}\left[ \textbf{y}_t \textbf{y}_{t+\tau}^\textnormal{H} \right]+  \mathbb{E}\left[ \textbf{y}_{t+\tau} \textbf{y}_t^\textnormal{H} \right]\right) = \boldsymbol{\Gamma} \textbf{A} \boldsymbol{\Lambda}_\tau \textbf{A}^\textnormal{H} \boldsymbol{\Gamma}^\textnormal{H} = \boldsymbol{\Lambda}_\tau,
\end{align*}
which is equivalent with,
\begin{align*}
 \boldsymbol{\Lambda}_\tau \textbf{A}^\textnormal{H} \boldsymbol{\Gamma}^\textnormal{H} =\textbf{A}^\textnormal{H} \boldsymbol{\Gamma}^\textnormal{H} \boldsymbol{\Lambda}_\tau.
\end{align*}
Since $\boldsymbol{\Lambda}_\tau$ has real-valued distinct diagonal entries and since $\boldsymbol{\Lambda}_\tau$ and $\textbf{A}^\textnormal{H} \boldsymbol{\Gamma}^\textnormal{H}$ commute, we get that $  \boldsymbol{\Gamma}\textbf{A}$ is also a diagonal matrix.  Consequently, $\boldsymbol{\Gamma}\textbf{A} $ is a unitary diagonal matrix, which implies that the diagonal elements of the matrix product are of the form $\exp(i\theta_1), \ldots, \exp(i\theta_d)$.

For the second part of the proof, let $\boldsymbol{\mu} =\boldsymbol{\mu}_{\textbf{x}}$ and   $\boldsymbol{\Gamma}\textbf{A} = \textbf{J}$, where $\textbf{J}$ is some phase-shift matrix. We next verify that  $\textbf{y}_\textnormal{\tiny{\textbullet}} \coloneqq g \circ \textbf{x}_\textnormal{\tiny{\textbullet}}$ is a solution, that is, we verify that $\textbf{y}_\textnormal{\tiny{\textbullet}}$ satisfies conditions (1)-(4) of Definition \ref{model:BSS}. Condition (1) is clearly satisfied, since $\mathbb{E}[\textbf{x}_t] = \boldsymbol{\mu}_{\textbf{x}}$. Furthermore, we have that the following holds for every $\tau,t\in\mathbb{N}$,
\begin{align*}
\ddot{\textbf{S}}_\tau[\textbf{y}_t] = \mathbb{E}\left[ \boldsymbol{\Gamma}\textbf{A} \textbf{z}_t  \left(\boldsymbol{\Gamma}\textbf{A} \textbf{z}_{t+\tau} \right)^\textnormal{H} \right] = \textbf{J} \ddot{\boldsymbol{\Lambda}}_\tau \textbf{J}^\textnormal{H},
\end{align*}
where $\ddot{\boldsymbol{\Lambda}}_0 = \textbf{I}_d$. Thus, conditions (2)-(3) are satisfied. Finally, we have for a fixed $\tau$ and for every $t\in\mathbb{N}$ that,
\begin{align*}
{\textbf{S}}_\tau[\textbf{y}_t] = \textbf{J} {\textbf{S}}_\tau[\textbf{z}_t]\textbf{J}^\textnormal{H} = \textbf{J} {\boldsymbol{\Lambda}}_\tau \textbf{J}^\textnormal{H} = {\boldsymbol{\Lambda}}_\tau,
\end{align*}
since diagonal matrices commute. Thus, condition (4) is satisfied and the proof is complete.
\end{proof}

\begin{proof}[Proof of Theorem \ref{bss:solutions}]
In order to simplify the notation, we denote $\textbf{S}_0 \coloneqq \textbf{S}_0\left[ \textbf{x}_t \right]$ and $\textbf{S}_\tau \coloneqq \textbf{S}_\tau\left[ \textbf{x}_t \right]$. First, let $g$ be a solution. Then, since $\textbf{x}_\textnormal{\tiny{\textbullet}}$ follows the mixing model, Lemma \ref{lma:permphase} gives,
\begin{align*}
\textbf{S}_0  \boldsymbol{\Gamma}^\textnormal{H} = \textbf{A} \textbf{A}^\textnormal{H} \boldsymbol{\Gamma}^\textnormal{H} = \textbf{A} \textbf{J}^\textnormal{H} \quad \Longrightarrow \quad \textbf{A} = \textbf{S}_0  \boldsymbol{\Gamma}^\textnormal{H} \textbf{J}.
\end{align*}
Using the above expressions, we get that,
\begin{align*}
\textbf{S}_\tau  \boldsymbol{\Gamma}^\textnormal{H}  =  \textbf{A} \boldsymbol{\Lambda}_\tau \textbf{A}^\textnormal{H} \boldsymbol{\Gamma}^\textnormal{H} =  \textbf{S}_0  \boldsymbol{\Gamma}^\textnormal{H} \textbf{J} \boldsymbol{\Lambda}_\tau  \textbf{J}^\textnormal{H} = \textbf{S}_0  \boldsymbol{\Gamma}^\textnormal{H} \boldsymbol{\Lambda}_\tau,
\end{align*}
and Equation \eqref{eq:sol1} follows by left-multiplying with $ \textbf{S}_0^{-1}$. In addition, Lemma \ref{lma:permphase} directly gives Equation \eqref{eq:sol2}. Since the process $\textbf{x}_\textnormal{\tiny{\textbullet}}$  is weakly stationary, the previous is true for every $t\in \mathbb{N}$.

Next, let Equations \eqref{eq:sol1} and \eqref{eq:sol2} be satisfied for every $t\in\mathbb{N}$. Left-multiplying with $\textbf{S}_0$, Equation \eqref{eq:sol1}  can be reformulated as,
\begin{align*}
\textbf{A} \boldsymbol{\Lambda}_\tau \textbf{A}^\textnormal{H} \boldsymbol{\Gamma}^\textnormal{H} = \textbf{A} \textbf{A}^\textnormal{H} \boldsymbol{\Gamma}^\textnormal{H} \boldsymbol{\Lambda}_\tau,
\end{align*}
which, after a left-multiplication with $\textbf{A}^{-1}$, gives that $ \textbf{A}^\textnormal{H}  \boldsymbol{\Gamma}^\textnormal{H}$ and $ \boldsymbol{\Lambda}_\tau$ commute. Since the diagonal matrix $ \boldsymbol{\Lambda}_\tau$ has real-valued distinct diagonal elements, we get that  $ \textbf{A}^\textnormal{H}   \boldsymbol{\Gamma}^\textnormal{H}$ is a diagonal matrix. Then, the scaling-equation gives that $\boldsymbol{\Gamma}\textbf{A}   \textbf{A}^\textnormal{H}   \boldsymbol{\Gamma}^\textnormal{H} = \textbf{I}_d$, i.e.,  $\boldsymbol{\Gamma}\textbf{A}$ is a unitary matrix. Consequently, $\boldsymbol{\Gamma}\textbf{A} = \textbf{J}$, where $\textbf{J}$ is some phase-shift matrix.  By  Lemma \ref{lma:permphase}, the functional $g$ is hereby a solution to the corresponding unmixing problem, and the proof is complete.
\end{proof}

\begin{proof}[Proof of Corollary \ref{cor:solution}]
Under the assumptions of Definition \ref{model:BSS}, we have that $\textbf{S}_0$ is a positive-definite matrix and that $\textbf{S}_\tau$ is a conjugate symmetric matrix. Thus, the matrix-square root of $\textbf{S}_0$ exists and the conjugate symmetric matrix-square root is unique, and consequently similar arguments as in the real-valued counterpart presented in \cite{ilmonen2012invariant} can be applied here. Hereby, $\textbf{S}_0^{-1/2} \textbf{S}_\tau \textbf{S}_0^{-1/2}$ can always be eigendecomposed.

The second part of the proof follows directly from Theorem \ref{bss:solutions}, by verifying that $\boldsymbol{\Gamma}  = \textbf{V}^\textnormal{H}\textbf{S}_0^{-1/2}$ satisfies Equations \eqref{eq:sol1} and \eqref{eq:sol2}.
\end{proof}

\begin{proof}[Proof of Lemma \ref{lma:affeqv}]
By Lemma \ref{lma:permphase}, we have that $\boldsymbol{\Gamma}\textbf{A} = \textbf{J}_1$ and $\tilde{\boldsymbol{\Gamma}}\textbf{CA} = \textbf{J}_2$, where $\textbf{J}_1$ and $\textbf{J}_2$ are some phase-shift matrices. Recall that the mixing matrix $\textbf{A}$ is nonsingular. Hereby, we get that $\boldsymbol{\Gamma} = \textbf{J}_1\textbf{A}^{-1}$ and $\tilde{\boldsymbol{\Gamma}}= \textbf{J}_2(\textbf{CA})^{-1}.$ By using the obtained expressions for $\boldsymbol{\Gamma} $ and  $\tilde{\boldsymbol{\Gamma}}$, we get that,
\begin{align*}
\tilde{\boldsymbol{\Gamma}}= \textbf{J}_2 \textbf{I}_d\textbf{A}^{-1}\textbf{C}^{-1} = \textbf{J}_2 \textbf{J}_1^\textnormal{H} \textbf{J}_1\textbf{A}^{-1}\textbf{C}^{-1} =  \textbf{J}_3 \boldsymbol{\Gamma} \textbf{C}^{-1},
\end{align*}
where $\textbf{J}_3 = \textbf{J}_2 \textbf{J}_1^\textnormal{H}$ is a phase-shift matrix, as the set of phase-shift matrices is stable under matrix multiplication.
\end{proof}

\begin{proof}[Proof of Lemma \ref{lma:samplesol}]
By assumption, the components of $\textbf{z}_\textnormal{\tiny{\textbullet}}$ have continuous marginals,  which implies that covariance matrix estimates are almost surely nonsingular. Thus, we can almost surely find a unique matrix  $\hat{\boldsymbol{\Sigma}}_0$, which is the conjugate symmetric inverse square root of $\hat{\textbf{S}}_0$.  We proceed to verify that $\hat{\boldsymbol{\Gamma}} =  \hat{\textbf{V}}^\textnormal{H} \hat{\boldsymbol{\Sigma}}_0$ satisfies the conditions of Definition \ref{def:estimating}. First,
\begin{align*}
 \hat{\textbf{V}}^\textnormal{H}\hat{\boldsymbol{\Sigma}}_0\hat{\textbf{S}}_0 \left( \hat{\textbf{V}}^\textnormal{H} \hat{\boldsymbol{\Sigma}}_0\right)^\textnormal{H} =  \hat{\textbf{V}}^\textnormal{H}  \hat{\textbf{V}}  = \textbf{I}_d.
\end{align*}
and similarly,
\begin{align*}
 \hat{\textbf{V}}^\textnormal{H} \hat{\boldsymbol{\Sigma}}_0\hat{\textbf{S}}_\tau \left( \hat{\textbf{V}}^\textnormal{H} \hat{\boldsymbol{\Sigma}}_0\right)^\textnormal{H} =
\hat{\textbf{V}}^\textnormal{H} \hat{\textbf{V}} \hat{\boldsymbol{\Lambda}}_\tau \hat{\textbf{V}}^\textnormal{H} \hat{\textbf{V}} = \hat{\boldsymbol{\Lambda}}_\tau.
\end{align*}
\end{proof}

 \begin{proof}[Proof of Lemma \ref{lma:affinvsample}]
By Lemma \ref{lma:samplesol}, we can write ${\hat{\boldsymbol{\Gamma}}} = \hat{\textbf{V}}^\textnormal{H} \hat{\boldsymbol{\Sigma}}_0$, where $\hat{\boldsymbol{\Sigma}}_0$ denotes the conjugate symmetric inverse square root of $\hat{\textbf{S}}_0.$ Note that the matrix-valued estimators are affine equivariant in the sense that, 
\begin{align*}
\hat{\textbf{S}}_j[\textbf{X} \textbf{B}^\top - \textbf{1}_T\textbf{b}^\top] =\textbf{B}\hat{\textbf{S}}_j[\textbf{X}]\textbf{B}^\textnormal{H},
\end{align*}
 $j \in\{0,\tau\}$, for all nonsingular $\mathbb{C}^{d\times d}$-matrices $\textbf{B}$ and all $\mathbb{C}^{d}$-vectors $\textbf{b}$. We proceed to verify that $\textbf{J} {\hat{\boldsymbol{\Gamma}}} \textbf{B}^{-1}$ satisfies the criteria of a finite sample solution, given in Definition \ref{def:estimating}. Using the affine equivariance property, we get that,
\begin{align*}
\textbf{J} {\hat{\boldsymbol{\Gamma}}} \textbf{B}^{-1} \hat{\textbf{S}}_0[\tilde{\textbf{X}}] \left(\textbf{J} \hat{\boldsymbol{\Gamma}} \textbf{B}^{-1}\right)^{\textnormal{H}} = \textbf{J}  \hat{\textbf{V}}^\textnormal{H} \hat{\boldsymbol{\Sigma}}_0\hat{\textbf{S}}_0[{\textbf{X}}] \hat{\boldsymbol{\Sigma}}_0 \hat{\textbf{V}}\textbf{J}^{\textnormal{H}} = \textbf{JJ}^{\textnormal{H}} = \textbf{I}_d,
\end{align*}
and,
\begin{align*}
\textbf{J} {\hat{\boldsymbol{\Gamma}}} \textbf{B}^{-1} \hat{\textbf{S}}_\tau[\tilde{\textbf{X}}] \left(\textbf{J} \hat{\boldsymbol{\Gamma}} \textbf{B}^{-1}\right)^{\textnormal{H}} = \textbf{J}  \hat{\textbf{V}}^\textnormal{H} \hat{\boldsymbol{\Sigma}}_0\hat{\textbf{S}}_\tau[{\textbf{X}}] \hat{\boldsymbol{\Sigma}}_0 \hat{\textbf{V}}\textbf{J}^{\textnormal{H}} = \textbf{J} \hat{\boldsymbol{\Lambda}}_\tau \textbf{J}^{\textnormal{H}} =\hat{\boldsymbol{\Lambda}}_\tau .
\end{align*}
Hereby, $\textbf{J} {\hat{\boldsymbol{\Gamma}}} \textbf{B}^{-1}$ is a finite sample solution for $\tilde{\textbf{X}}$ and the proof is complete.
\end{proof}

\begin{proof}[Proof of Lemma \ref{lma:consistent}]
Denote $\hat{\textbf{S}}_j \coloneqq \hat{\textbf{S}}_j[\textbf{X} ] $ and ${\textbf{S}}_j \coloneqq {\textbf{S}}_j[\textbf{x}_t ]$, where $j \in \{0,\tau\}$. By Lemmas \ref{lma:affeqv} and \ref{lma:affinvsample}, it is sufficient to consider the trivial mixing scenario $\textbf{A} = \textbf{I}_d$. Note that $\gamma_T (\hat{\textbf{J}}\hat{\boldsymbol{\Gamma}}-\boldsymbol{\Gamma})=\gamma_T (\hat{\textbf{J}}\hat{\boldsymbol{\Gamma}}\boldsymbol{\Gamma}^{-1}-\textbf{I}_d)\boldsymbol{\Gamma}$ and by Lemma \ref{lma:affinvsample}, we have that $\hat{\textbf{J}}\hat{\boldsymbol{\Gamma}}\boldsymbol{\Gamma}^{-1}$ is a finite sample solution for $\textbf{X}\boldsymbol{\Gamma}^\top$. Hereby, the trivial mixing can be generalized to any full-rank mixing matrix $\textbf{A}$.

 Under trivial mixing, we have that $\alpha_T(\hat{\textbf{S}}_0 - \textbf{I}_d) = \mathcal{O}_p(1)$ and $\beta_T(\hat{\textbf{S}}_\tau - \boldsymbol{\Lambda}_\tau ) = \mathcal{O}_p(1)$. Furthermore, by Corollary \ref{cor:solution} and Lemma \ref{lma:samplesol}, we have $\hat{\boldsymbol{\Gamma}} =  \hat{\boldsymbol{V}}^\textnormal{H}\hat{\boldsymbol{\Sigma}}_0$ and ${\boldsymbol{\Gamma}} =  {\boldsymbol{V}}^\textnormal{H}{\boldsymbol{\Sigma}}_0$, where $\hat{\boldsymbol{\Sigma}}_0 = \hat{\textbf{S}}_0^{-1/2}$ and ${\boldsymbol{\Sigma}}_0 = {\textbf{S}}_0^{-1/2} $ are conjugate symmetric. Note that under trivial mixing, we have ${\textbf{S}}_0^{-1/2}  = \textbf{I}_d$ and $\textbf{V} = \textbf{J}$, where $\textbf{J}$ is some phase-shift matrix.

We next denote $\hat{\boldsymbol{\Sigma}}_\tau =\hat{\boldsymbol{\Sigma}}_0 \hat{\textbf{S}}_\tau  \hat{\boldsymbol{\Sigma}}_0$ and show that $\gamma_T(\hat{\boldsymbol{\Sigma}}_\tau - \boldsymbol{\Lambda}_\tau) = \mathcal{O}_p(1)$. The uniform tightness follows from Lemma \ref{lma:inverse}, Slutsky's Lemma and the following factorization,
\begin{align*}
\gamma_T (\hat{\boldsymbol{\Sigma}}_\tau - \boldsymbol{\Lambda}_\tau) &= \gamma_T\left[( \hat{\boldsymbol{\Sigma}}_0 - \textbf{I}_d  )( \hat{\textbf{S}}_\tau - \boldsymbol{\Lambda}_\tau)( \hat{\boldsymbol{\Sigma}}_0 - \textbf{I}_d ) + ( \hat{\boldsymbol{\Sigma}}_0 - \textbf{I}_d  )( \hat{\textbf{S}}_\tau - \boldsymbol{\Lambda}_\tau) \right. \\
&+ ( \hat{\boldsymbol{\Sigma}}_0 - \textbf{I}_d  )\boldsymbol{\Lambda}_\tau( \hat{\boldsymbol{\Sigma}}_0 - \textbf{I}_d  ) + ( \hat{\boldsymbol{\Sigma}}_0 - \textbf{I}_d  )\boldsymbol{\Lambda}_\tau +  ( \hat{\textbf{S}}_\tau - \boldsymbol{\Lambda}_\tau)  \\
&+ \left. ( \hat{\textbf{S}}_\tau - \boldsymbol{\Lambda}_\tau)( \hat{\boldsymbol{\Sigma}}_0 - \textbf{I}_d  ) + \boldsymbol{\Lambda}_\tau( \hat{\boldsymbol{\Sigma}}_0 - \textbf{I}_d  ) \right]\\
&= o_p(1) + \gamma_T \left[( \hat{\boldsymbol{\Sigma}}_0 - \textbf{I}_d  )\boldsymbol{\Lambda}_\tau +  ( \hat{\textbf{S}}_\tau - \boldsymbol{\Lambda}_\tau) + \boldsymbol{\Lambda}_\tau( \hat{\boldsymbol{\Sigma}}_0 - \textbf{I}_d  ) \right]\\
&= \mathcal{O}_p(1).
\end{align*}
Next, recall the eigendecomposition $\hat{\boldsymbol{\Sigma}}_\tau = \hat{\textbf{V}} \hat{\boldsymbol{\Lambda}}_\tau \hat{\textbf{V}}^\textnormal{H}$ and that the space of unitary matrices is compact. Consequently, $\textbf{U} = \mathcal{O}_p(1)$ for any unitary $\textbf{U}$. By right-multiplying both sides of the eigendecomposition with $\hat{\boldsymbol{V}}$ and by subtracting $\boldsymbol{\Lambda}_\tau$ from both sides, we get,
\begin{align*}
\hat{\boldsymbol{\Sigma}}_\tau  \hat{\textbf{V}} - {\boldsymbol{\Lambda}}_\tau = \hat{\textbf{V}} \hat{\boldsymbol{\Lambda}}_\tau - {\boldsymbol{\Lambda}}_\tau,
\end{align*}
where both sides of the equation can be further factorized as follows,
\begin{align*}
(\hat{\boldsymbol{\Sigma}}_\tau - \boldsymbol{\Lambda}_\tau)\hat{\textbf{V}} +    {\boldsymbol{\Lambda}}_\tau ( \hat{\textbf{V}} - \textbf{I}_d) = (\hat{\textbf{V}} - \textbf{I}_d )\hat{\boldsymbol{\Lambda}}_\tau + (\hat{\boldsymbol{\Lambda}}_\tau  - {\boldsymbol{\Lambda}}_\tau),
\end{align*}
and by multiplying with $\gamma_T$ and by rearranging the terms, we obtain,
\begin{align}
\label{eq:evals}
\gamma_T  \left[{\boldsymbol{\Lambda}}_\tau ( \hat{\textbf{V}} - \textbf{I}_d) - (\hat{\textbf{V}} - \textbf{I}_d )\hat{\boldsymbol{\Lambda}}_\tau - (\hat{\boldsymbol{\Lambda}}_\tau  - {\boldsymbol{\Lambda}}_\tau)\right] =\mathcal{O}_p(1).
 \end{align}
By assumption, the diagonal matrix ${\boldsymbol{\Lambda}}_\tau$ has distinct real-valued diagonal elements. Furthermore, since the matrix $\hat{\boldsymbol{\Lambda}}_\tau$ is obtained by estimating an eigendecomposition, it is also diagonal, with diagonal elements denoted as $\hat{\lambda}_\tau^{(1)},\ldots, \hat{\lambda}_\tau^{(d)}$. Then, consider the element $(j,k)$, $j\neq k$, of Equation \eqref{eq:evals},
\begin{align}
\label{eq:evals2}
\gamma_T \hat{\textbf{V}}_{jk}({\lambda}_\tau^{(j)} -  \hat{\lambda}_\tau^{(k)})  =\mathcal{O}_p(1).
 \end{align}
Since, $\gamma_T(\hat{\boldsymbol{\Sigma}}_\tau - \boldsymbol{\Lambda}_\tau) = \mathcal{O}_p(1)$, we get that $\hat{\lambda}_\tau^{(k)}$ converges in probability to ${\lambda}_\tau^{(k)}$. Furthermore, since the diagonal elements of $\boldsymbol{\Lambda}_\tau$ are distinct, we can divide both sides of Equation \eqref{eq:evals2} with $\lambda_\tau^{(j)} - \hat{\lambda}_\tau^{(k)}$, which gives us that $\gamma_T[\hat{\textbf{V}} ]_{jk} = \mathcal{O}_p(1)$ holds asymptotically when $j\neq k$.

Next, let $\hat{\textbf{J}}$ be a $T$-indexed sequence of phase-shift matrices, such that the diagonal entries of the product  $\hat{\textbf{V}}\hat{\textbf{J}}^\textnormal{H} $ are in the positive real-axis for every $T\in \mathbb{N}$. Note that any complex-number can be expressed in the form $r\exp(i\theta)$, where $\theta$ is the phase and $r$ is the modulus, i.e., the length of the complex number. In the matrix case, we can similarly express $\hat{\textbf{V}}$ such that $\hat{\textbf{V}} = \hat{\textbf{V}}_{r} \hat{\textbf{V}}_{\theta}$, where $\hat{\textbf{V}}_{\theta}$ is a phase-shift matrix that contains the phases of the diagonal diagonal elements of $\hat{\textbf{V}}$ and consequently $\hat{\textbf{V}}_{r} $ is a matrix with real-valued diagonal entries. Then, given a phase-shift matrix $\hat{\textbf{V}}_{\theta} = \textnormal{diag}(\exp(i\theta_1,\ldots,i\theta_d))$, we can always construct a matrix $\hat{\textbf{J}}^\textnormal{H} = \textnormal{diag}(\exp(-i\theta_1,\ldots,-i\theta_d))$ such that the product $\hat{\textbf{V}}\hat{\textbf{J}}^\textnormal{H} $ has real-valued diagonal entries.

After the diagonal entries have been rotated to the positive real-axis, the rotated diagonal elements are equal to the corresponding moduli, that is, $\hat{\textbf{V}}_{kk}\hat{\textbf{J}}_{kk}^\textnormal{H} = | \hat{\textbf{V}}_{kk}|$. Then, since $\hat{\textbf{V}}$ is unitary, each of its row vectors has length one and the absolute value of a single element is at most one, which gives us,
\begin{align*}
 \left| 1 -  \hat{\textbf{V}}_{kk} \hat{\textbf{J}}^\textnormal{H}_{kk} \right|=1 -  \hat{\textbf{V}}_{kk} \hat{\textbf{J}}^\textnormal{H}_{kk} = 1 - \left|\hat{\textbf{V}}_{kk}\right| = 1 - \sqrt{1-\sum_{\substack{h=1 \\ h\neq k}}^d \left|\hat{\textbf{V}}_{hk}\right|^2}.
\end{align*}
By unitarity of $\hat{\textbf{V}}$, the above square root is always between zero and one, and thus squaring the square root produces a smaller or equal number. Hereby, using the (asymptotic) uniform tightness of the off-diagonal elements, we get that asymptotically,
\begin{align*}
 \left| 1 -  \hat{\textbf{V}}_{kk} \hat{\textbf{J}}^\textnormal{H}_{kk} \right| \leq 1  -{1+\sum_{\substack{h=1 \\ h\neq k}}^d \left|\hat{\textbf{V}}_{hk}\right|^2} = \sum_{\substack{h=1 \\ h\neq k}}^d \left|\hat{\textbf{V}}_{hk}\right|^2 = \mathcal{O}_p(1/\gamma_T^2).
\end{align*}
Hereby, by combining the results for the diagonal and off-diagonal elements, we get that there exists a sequence of phase-shift matrices $\hat{\textbf{J}}$ such that $\gamma_T(\hat{\textbf{V}}\hat{\textbf{J}}^\textnormal{H} - \textbf{I}_d) = \mathcal{O}_p(1)$ holds asymptotically.

The claim of the lemma can then be written as,
\begin{align*}
\gamma_T \left[\hat{\textbf{J}} \hat{\boldsymbol{\Gamma}} - \textbf{I}_d \right] &=
\gamma_T\left[\left((\hat{\textbf{V}} \hat{\textbf{J}}^\textnormal{H}\right)^\textnormal{H}
\hat{\boldsymbol{\Sigma}}_0- \textbf{I}_d \right]\\
&=\left(\hat{\textbf{V}} \hat{\textbf{J}}^\textnormal{H}\right)^\textnormal{H} \gamma_T(\hat{\boldsymbol{\Sigma}}_0 - \textbf{I}_d ) +\gamma_T \left( \hat{\textbf{V}} \hat{\textbf{J}}^\textnormal{H} - \textbf{I}_d \right)^\textnormal{H}\\
&= \mathcal{O}_p(1) \mathcal{O}_p(1) +  \mathcal{O}_p(1) = \mathcal{O}_p(1) .
\end{align*}
\end{proof}

\begin{proof}[Proof of Theorem \ref{thm:limiting}]
Let $\hat{\textbf{S}}_j\coloneqq \hat{\textbf{S}}_j[\textbf{X} ]$, $\tau \in\{0,\tau\}$, and recall the equations from Definition \ref{def:estimating},
\begin{align}
\label{proof:eqs}
\hat{\boldsymbol{\Gamma} }  \hat{\textbf{S}}_0 \hat{\boldsymbol{\Gamma}}^\textnormal{H}  =\textbf{I}_d \qquad \textnormal{ and} \qquad \hat{\boldsymbol{\Gamma} }  \hat{\textbf{S}}_\tau \hat{\boldsymbol{\Gamma}}^\textnormal{H}  =\hat{\boldsymbol{\Lambda}}_\tau.
\end{align}
In order to simplify the notation, we denote $\hat{\textbf{G}} \coloneqq \hat{\textbf{J}}\hat{\boldsymbol{\Gamma}},$ where $\hat{\textbf{J}}$ is a $T$-indexed sequence of phase-shift matrices such that the diagonal elements of $ \hat{\textbf{J}}\hat{\boldsymbol{\Gamma}}$ are on the positive real-axis.   Under Lemma \ref{lma:consistent}, we have that $\gamma_T(\hat{\textbf{G}}  - \textbf{I}_d) = \mathcal{O}_p(1)$ and note that both parts of Equation \eqref{proof:eqs} also hold for $\hat{\textbf{G}}$.

The left-part of Equation \eqref{proof:eqs} can then be expanded as,
\begin{align*}
 &( \hat{\textbf{G}}  - \textbf{I}_d ) \hat{\textbf{S}}_0 \hat{\textbf{G}} ^\textnormal{H} + (\hat{\textbf{S}}_0  - \textbf{I}_d ) \hat{\textbf{G}}^\textnormal{H} +  ( \hat{\textbf{G}}^\textnormal{H} - \textbf{I}_d ) = 0,
\end{align*}
where the left-hand-side can be further expanded as,
\begin{align*}
 &( \hat{\textbf{G}} - \textbf{I}_d )( \hat{\textbf{S}}_0 - \textbf{I}_d ) (\hat{\textbf{G}}^\textnormal{H} - \textbf{I}_d ) +
( \hat{\textbf{G}} - \textbf{I}_d )( \hat{\textbf{S}}_0 - \textbf{I}_d )
+( \hat{\textbf{G}}- \textbf{I}_d )( \hat{\textbf{G}}^\textnormal{H} - \textbf{I}_d )
\\
 &+ ( \hat{\textbf{G}} - \textbf{I}_d )+(\hat{\textbf{S}}_0 - \textbf{I}_d )( \hat{\textbf{G}}^\textnormal{H} - \textbf{I}_d )
+(\hat{\textbf{S}}_0 - \textbf{I}_d ) +  ( \hat{\textbf{G}}^\textnormal{H} - \textbf{I}_d )\\
& = ( \hat{\textbf{G}}- \textbf{I}_d ) +   (\hat{\textbf{S}}_0 - \textbf{I}_d ) +  ( \hat{\textbf{G}}^\textnormal{H} - \textbf{I}_d )+ \mathcal{O}_p(1/\gamma^2_T).
\end{align*}
By rearranging the terms, we get the following form,
\begin{align}
\label{eq1:appx}
 ( \hat{\textbf{G}}^\textnormal{H} - \textbf{I}_d  ) = ( \textbf{I}_d - \hat{\textbf{G}}  )+  (\textbf{I}_d -\hat{\textbf{S}}_0  )  + \mathcal{O}_p(1/\gamma^2_T).
\end{align}
Similarly, the right-part of Equation \eqref{proof:eqs} can be expanded as,
\begin{align*}
( \hat{\textbf{G}}- \textbf{I}_d)  \hat{\textbf{S}}_\tau\hat{\textbf{G}}^\textnormal{H}  + ( \hat{\textbf{S}}_\tau - \boldsymbol{\Lambda}_\tau)  \hat{\textbf{G}}^\textnormal{H}  +   \boldsymbol{\Lambda}_\tau( \hat{\textbf{G}}^\textnormal{H} - \textbf{I}_d) +(  \boldsymbol{\Lambda}_\tau - \hat{\boldsymbol{\Lambda}}_\tau ) = 0,
\end{align*}
where the left-hand-side can be further expanded as,
\begin{align*}
&( \hat{\textbf{G}}- \textbf{I}_d) ( \hat{\textbf{S}}_\tau-  \boldsymbol{\Lambda}_\tau ) (  \hat{\textbf{G}}^\textnormal{H} - \textbf{I}_d ) +
( \hat{\textbf{G}} - \textbf{I}_d) ( \hat{\textbf{S}}_\tau -  \boldsymbol{\Lambda}_\tau )
\\
&+ ( \hat{\textbf{G}} - \textbf{I}_d)  \boldsymbol{\Lambda}_\tau    (\hat{\textbf{G}}^\textnormal{H} - \textbf{I}_d)
 + ( \hat{\textbf{G}} - \textbf{I}_d)  \boldsymbol{\Lambda}_\tau
+( \hat{\textbf{S}}_\tau - \boldsymbol{\Lambda}_\tau)  (\hat{\textbf{G}}^\textnormal{H} - \textbf{I}_d ) \\
&+ ( \hat{\textbf{S}}_\tau - \boldsymbol{\Lambda}_\tau)
 +   \boldsymbol{\Lambda}_\tau( \hat{\textbf{G}}^\textnormal{H} - \textbf{I}_d ) +( \boldsymbol{\Lambda}_\tau -  \hat{\boldsymbol{\Lambda}}_\tau )  \\
&=(\hat{\textbf{G}} - \textbf{I}_d)  \boldsymbol{\Lambda}_\tau  +   ( \hat{\textbf{S}}_\tau - \boldsymbol{\Lambda}_\tau)  +  \boldsymbol{\Lambda}_\tau( \hat{\textbf{G}}^\textnormal{H} - \textbf{I}_d )+( \boldsymbol{\Lambda}_\tau -  \hat{\boldsymbol{\Lambda}}_\tau ) +\mathcal{O}_p(1/\gamma^2_T).
\end{align*}
Hereby, we obtain,
\begin{align*}
( \hat{\textbf{G}}  - \textbf{I}_d  )  \boldsymbol{\Lambda}_\tau  +  (  \hat{\textbf{S}}_\tau -\boldsymbol{\Lambda}_\tau)       +  \boldsymbol{\Lambda}_\tau(   \hat{\textbf{G}} ^\textnormal{H}-\textbf{I}_d  ) = (  \hat{\boldsymbol{\Lambda}}_\tau  - \boldsymbol{\Lambda}_\tau  )  +\mathcal{O}_p(1/\gamma^2_T),
\end{align*}
which is, by using the expression for $\hat{\textbf{G}}^\textnormal{H}-\textbf{I}_d $ given by  Equation \eqref{eq1:appx}, equivalent with,
\begin{align}
\label{eq2:appx}
\hat{\textbf{G}} \boldsymbol{\Lambda}_\tau - \boldsymbol{\Lambda}_\tau\hat{\textbf{G}}    &=   ( \boldsymbol{\Lambda}_\tau-  \hat{\textbf{S}}_\tau )+ \boldsymbol{\Lambda}_\tau (\hat{\textbf{S}}_0 -  \textbf{I}_d ) +  (  \hat{\boldsymbol{\Lambda}}_\tau  - \boldsymbol{\Lambda}_\tau  )+ \mathcal{O}_p(1/\gamma^2_T).
\end{align}

Recall that $\hat{\textbf{G}} = \hat{\textbf{J}}\hat{\boldsymbol{\Gamma}}$ and that $\boldsymbol{\Lambda}_\tau$, $\hat{\boldsymbol{\Lambda}}_\tau$ are both diagonal matrices with real-valued diagonal elements.
Then, by rearranging the terms and by considering the element $(j,j)$ of Equation \eqref{eq1:appx}, we get that
\begin{align*}
\hat{\textbf{J}}_{jj}\hat{\boldsymbol{\Gamma}}_{jj} - 1 = \frac{1}{2}\left(1 -  \left[\hat{\textbf{S}}_0\right]_{jj} \right) + \mathcal{O}_p(1/\gamma^2_T).
\end{align*}
Similarly, by considering the element $(j,k)$, $j\neq k$, of Equation \eqref{eq2:appx}, we get that
\begin{align*}
(\lambda^{(k)}_\tau - \lambda^{(j)}_\tau) \hat{\textbf{J}}_{jj}\hat{\boldsymbol{\Gamma}} _{jk} = \lambda^{(j)}_\tau \left[\hat{\textbf{S}}_0\right]_{jk} - \left[\hat{\textbf{S}}_\tau\right]_{jk} + \mathcal{O}_p(1/\gamma^2_T).
\end{align*}
The theorem then follows by multiplying both sides with $\gamma_T$.
\end{proof}

\begin{proof}[Proof of Theorem \ref{thm:covconvergence}]
In order to incorporate the shift $\tau$ given in Definition \ref{model:BSS}, we first introduce some notation. For $\boldsymbol{\eta}_{\textnormal{\tiny{\textbullet}}}$ given in Assumption \ref{ass:summablecov}, we set $ \textbf{y}^\top_\textnormal{\tiny{\textbullet}} = \begin{pmatrix}\boldsymbol{\eta}_{\textnormal{\tiny{\textbullet}}}^\top & \boldsymbol{\eta}_{\textnormal{\tiny{\textbullet}}+\tau}^\top \end{pmatrix}$. We also introduce functions $g_j(\textbf{y}_\textnormal{\tiny{\textbullet}}):\mathbb{R}^{2\ell}\rightarrow \mathbb{R},  j\in\{1,2,\ldots,4d\},$ through connections $g_j(\textbf{y}_\textnormal{\tiny{\textbullet}}) = \tilde{f}_j(\boldsymbol{\eta}_{\textnormal{\tiny{\textbullet}}}) ,j\in\{1,2,\ldots,2d\}$ and $g_{j+2d}(\textbf{y}_\textnormal{\tiny{\textbullet}}) = \tilde{f}_{j}(\boldsymbol{\eta}_{\textnormal{\tiny{\textbullet}}+\tau}), j\in\{1,2,\ldots,2d\}$. That is, for $j\geq 2d+1$ the function $g_j$ corresponds to $\tilde{f}_j$ evaluated at shift $\boldsymbol{\eta}_{\textnormal{\tiny{\textbullet}}+\tau}$.

The unsymmetrized autocovariance matrix estimator with lag $\tau$ is defined as,
\begin{align*}
\tilde{\textbf{S}}_\tau[\textbf{X}] = \frac{1}{T-\tau} \sum_{t=1}^{T-\tau}\left( \textbf{X}_t - \hat{\boldsymbol{\mu}} \right) \left( \textbf{X}_{t+\tau}- \hat{\boldsymbol{\mu}} \right)^\textnormal{H},
\end{align*}
where $\hat{\boldsymbol{\mu}} \coloneqq \hat{\boldsymbol{\mu}}[\textbf{X}]$. Let $\tilde{\textbf{X}} = \textbf{X} - \textbf{1}_T\boldsymbol{\mu}^\top$ and $\tilde{\boldsymbol{\mu}} =  \hat{\boldsymbol{\mu}}-\boldsymbol{\mu}$. Under Assumption \ref{ass:summablecov}, we have that the $k$th component $\tilde{\boldsymbol{X}}_t^{(k)}$ has the same asymptotic autocovariance function as $g_k(\textbf{y}_t) + ig_{k+d}(\textbf{y}_t)  - \mathbb{E}[g_k(\textbf{y}_t) + ig_{k+d}(\textbf{y}_t)]$. To improve the fluency of the proof, we denote $f_k(\textbf{y}_t) \coloneqq g_k(\textbf{y}_t) -   \mathbb{E}[g_k(\textbf{y}_t)]$.

We can  reformulate the  estimator as follows,
\begin{align*}
\tilde{\textbf{S}}_\tau[\textbf{X}] =\frac{1}{T-\tau}\sum_{t=1}^{T-\tau}&\left[\tilde{\textbf{X}}_t\tilde{\textbf{X}}^\textnormal{H}_{t+\tau} - \tilde{\textbf{X}}_t\tilde{\boldsymbol{\mu}}^\textnormal{H}-\tilde{\boldsymbol{\mu}}\tilde{\textbf{X}}^\textnormal{H}_{t+\tau} +\tilde{\boldsymbol{\mu}}\tilde{\boldsymbol{\mu}}^\textnormal{H} \right],
\end{align*}
where the last three terms of the sum are equal to,
\begin{align}
\label{eq:mean}
\frac{1}{T-\tau}\left[ \sum_{t=T-\tau+1}^T \tilde{\textbf{X}}_t\tilde{\boldsymbol{\mu}}^\textnormal{H} +\sum_{t=1}^\tau \tilde{\boldsymbol{\mu}}\tilde{\textbf{X}}_t^\textnormal{H} - (T + \tau) \tilde{\boldsymbol{\mu}}\tilde{\boldsymbol{\mu}}^\textnormal{H}\right] = \mathcal{O}_p(1/T).
\end{align}
Equation \eqref{eq:mean} holds, since the first two terms are finite sums. Furthermore, since $\mathbb{E}[\tilde{\boldsymbol{\mu}}] = \textbf{0}$ we get that, under Assumption \ref{ass:summablecov}, the $k$th component $\tilde{\boldsymbol{\mu}}^{(k)}$ is asymptotically equivalent with  $(1/T) \sum_{t=1}^T[f_k(\textbf{y}_t) - \mathbb{E}[ f_k(\textbf{y}_t)]+ i(f_{(k+d)}(\textbf{y}_t) -\mathbb{E}[ f_{k+d}(\textbf{y}_t)])]$. We can then directly apply Theorem \ref{thm:breuer-major2}, which in combination with Prohorov's theorem and Corollary \ref{cor:eqvrealcomplex}, gives that $\tilde{\boldsymbol{\mu}} = \mathcal{O}_p(1/\sqrt{T})$ and consequently, the last term of Equation \eqref{eq:mean} is $\mathcal{O}_p(1/T)$. 

The symmetrized autocovariance estimator can then be expressed as,
\begin{align*}
\hat{\textbf{S}}_\tau[\textbf{X}] = \frac{1}{2(T-\tau)} \sum_{t=1}^{T-\tau}\left( \tilde{\textbf{X}}_t\tilde{\textbf{X}}^\textnormal{H}_{t+\tau} + \tilde{\textbf{X}}_{t+\tau}\tilde{\textbf{X}}^\textnormal{H}_{t} \right) + \mathcal{O}_p(1/T),
\end{align*}
and the existence of fourth moments and the model assumptions given in Definition \ref{model:BSS} give us that,
\begin{align*}
\mathbb{E}\left[\hat{\textbf{S}}_\tau[\textbf{X}] \right] = \frac{1}{2}\left(\ddot{\boldsymbol{\Lambda}}_\tau + \ddot{\boldsymbol{\Lambda}}_\tau^\textnormal{H}\right) + \mathcal{O}(1/T),
\end{align*}
where $\ddot{\boldsymbol{\Lambda}}_0 = \textbf{I}_d$. Hereby, we have that,
\begin{align*}
\sqrt{T}\left(\hat{\textbf{S}}_0[\textbf{X}] - \textbf{I}_d\right) &=  \frac{1}{\sqrt{T}}\sum_{t=1}^{T}\left( \tilde{\textbf{X}}_t\tilde{\textbf{X}}^\textnormal{H}_{t} -  \mathbb{E}\left[ \tilde{\textbf{X}}_t\tilde{\textbf{X}}^\textnormal{H}_{t} \right] \right) + o_p(1).
\end{align*}
Note that asymptotically, it is indifferent whether we scale the autocovariance estimators with $1/T$ or $1/(T-\tau)$. Thus, for the fixed $\tau$, that satisfies Definition \ref{model:BSS},  we have that $\mathbb{E}[\boldsymbol{\Lambda}_\tau \hat{\textbf{S}}_0[\textbf{X}] - \hat{\textbf{S}}_\tau[\textbf{X}]] = \mathcal{O}(1/T)$,  which gives that,
\begin{align*}
\sqrt{T}\left(\boldsymbol{\Lambda}_\tau \hat{\textbf{S}}_0[\textbf{X}] - \hat{\textbf{S}}_\tau[\textbf{X}] \right)
\end{align*}
is equal to
\begin{align*}
\frac{1}{\sqrt{T}}\sum_{t=1}^{T}& \frac{1}{2}\left( 2\boldsymbol{\Lambda}_\tau  \tilde{\textbf{X}}_t\tilde{\textbf{X}}^\textnormal{H}_{t} - \tilde{\textbf{X}}_t\tilde{\textbf{X}}^\textnormal{H}_{t+\tau} - \tilde{\textbf{X}}_{t+\tau}\tilde{\textbf{X}}^\textnormal{H}_{t} \right. \\
&\left. - \mathbb{E}\left[2\boldsymbol{\Lambda}_\tau  \tilde{\textbf{X}}_t\tilde{\textbf{X}}^\textnormal{H}_{t} - \tilde{\textbf{X}}_t\tilde{\textbf{X}}^\textnormal{H}_{t+\tau} - \tilde{\textbf{X}}_{t+\tau}\tilde{\textbf{X}}^\textnormal{H}_{t}  \right]\right) + o_p(1),
\end{align*}
such that in the above expression the terms $\tilde{\textbf{X}}_{k+\tau} = \textbf{0}$, when $k+\tau > T $.

Let $\hat{\textbf{J}}$ be the $T$-indexed sequence that sets the diagonal elements of $\hat{\boldsymbol{\Gamma}}$ to the positive real-axis. By Theorem \ref{thm:limiting} and Assumption \ref{ass:summablecov}, we have that the diagonal element $(j,j)$ of $\sqrt{T}(  \hat{\textbf{J}}\hat{\boldsymbol{\Gamma}} - \textbf{I}_d)$ is asymptotically equivalent with $\textbf{H}_{jj} + o_p(1)$, defined as,
\begin{align*}
\textbf{H}_{jj} = \frac{1}{2\sqrt{T}} \sum_{t=1}^T \left( h_{j,j}(\textbf{y}_t) - \mathbb{E}\left[ h_{j,j}(\textbf{y}_t)  \right] + i\left(\tilde{h}_{j,j}(\textbf{y}_t) - \mathbb{E}\left[ \tilde{h}_{j,j}(\textbf{y}_t)  \right] \right)   \right),
\end{align*}
such that for $j\in\{1,\ldots d\}$ we have that $h_{j,j}(\textbf{y}_t) =   \left(f_j(\textbf{y}_t)\right)^2 + \left(f_{j+d}(\textbf{y}_t)\right)^2$ and $\tilde{h}_{j,j}(\textbf{y}_t) =  0$.

 The off-diagonal element $(j,k)$, $j\neq k$, of $\sqrt{T}(  \hat{\textbf{J}}\hat{\boldsymbol{\Gamma}} - \textbf{I}_d)$ is asymptotically equivalent with $\textbf{H}_{jk} + o_p(1)$, where $\textbf{H}_{jk}$ is equal to,
\begin{align*}
\frac{1}{ \sqrt{T} \left(\lambda^{(k)}_\tau - \lambda^{(j)}_\tau\right)}\sum_{t=1}^T \left( {h}_{j,k}(\textbf{y}_t) - \mathbb{E}\left[  {h}_{j,k}(\textbf{y}_t) \right] +  i\left(\tilde{h}_{j,k}(\textbf{y}_t)  - \mathbb{E}\left[  \tilde{h}_{j,k}(\textbf{y}_t) \right] \right)  \right),
\end{align*}
where,
\begin{align*}
{h}_{j,k}(\textbf{y}_t) =& \lambda_\tau^{(j)}\left[ f_j(\textbf{y}_t)f_k(\textbf{y}_t) + f_{j+d}(\textbf{y}_t)f_{k+d}(\textbf{y}_t)   \right]
- \frac{1}{2} \left[ f_{j}(\textbf{y}_t)f_{k+2d}(\textbf{y}_t) \right. \\
& \left. +f_{j+d}(\textbf{y}_t) f_{k+3d}(\textbf{y}_t) + f_{j+2d}(\textbf{y}_t) f_{k}(\textbf{y}_t) + f_{j+3d}(\textbf{y}_t) f_{k+d}(\textbf{y}_t)\right],
\end{align*}
and,
\begin{align*}
\tilde{h}_{j,k}(\textbf{y}_t) =& \lambda_\tau^{(j)}\left[ f_{j+d}(\textbf{y}_t)f_k(\textbf{y}_t) - f_{j}(\textbf{y}_t)f_{k+d}(\textbf{y}_t)   \right]
-\frac{1}{2}\left[ f_{j+d}(\textbf{y}_t)f_{k+2d}(\textbf{y}_t) \right.\\
& \left. + f_{j+3d}(\textbf{y}_t) f_{k}(\textbf{y}_t) - f_{j}(\textbf{y}_t) f_{k+3d}(\textbf{y}_t) - f_{j+2d}(\textbf{y}_t) f_{k+d}(\textbf{y}_t)\right].
\end{align*}
Let $\mathfrak{Re}[\textbf{H}]$ be the real-part and let $\mathfrak{Im}[\textbf{H}]$ be the imaginary-part of the $\mathbb{C}^{d\times d}$-valued matrix $\textbf{H}$ and let $\textbf{v} = \textnormal{vec}\begin{pmatrix}\mathfrak{Re}[\textbf{H}] &  \mathfrak{Im}[\textbf{H}] \end{pmatrix}$, such that the $\mathbb{C}^{2d^2}$-valued vector $\textbf{v}$ first contains the columns of the real-part and then the columns of the imaginary-part.

Under Assumption \ref{ass:summablecov}, we have that $f_{j}(\textbf{y}_t)$ has finite fourth moments for every $j \in\{1,\ldots 4d\}$ and every $t\in\mathbb{N}$. Hereby, Cauchy-Schwarz inequality gives that every $h_{j,k}(\textbf{y}_t)$ and $\tilde{h}_{j,k}(\textbf{y}_t)$ are square-integrable for every $t\in \mathbb{N}$. Furthermore, covariances and cross-covariances of the Gaussian process $\textbf{y}_\textnormal{\tiny{\textbullet}}$ are summable. Hereby, we can apply Theorem \ref{thm:breuer-major2} for $\textbf{v}$, which gives that,
\begin{align*}
\textbf{v} \xrightarrow[T\rightarrow \infty ]{\mathcal{D}} \boldsymbol{\rho} \sim \mathcal{N}_{2d^2}(\textbf{0}, \boldsymbol{\Sigma}_{\boldsymbol{\rho}}), \quad \textnormal{ where } \quad \boldsymbol{\Sigma}_{\boldsymbol{\rho}} = \begin{pmatrix} \boldsymbol{\Sigma}_{\boldsymbol{\rho}_1} & \boldsymbol{\Sigma}_{\boldsymbol{\rho}_{12}} \\ \boldsymbol{\Sigma}_{\boldsymbol{\rho}_{12}}^\top & \boldsymbol{\Sigma}_{\boldsymbol{\rho}_2} \end{pmatrix}.
\end{align*}
Corollary \ref{cor:eqvrealcomplex} then gives us,
\begin{align*}
\sqrt{T} \cdot \textnormal{vec}\left(\xhatt{\textbf{J}}\hat{\boldsymbol{\Gamma}} - \textbf{I}_d\right) \xrightarrow[T\rightarrow \infty]{\mathcal{D}} \boldsymbol{\nu} \sim \mathcal{N}_{d^2}(\textbf{0}, \boldsymbol{\Sigma}_{\boldsymbol{\nu}}, \textbf{P}_{\boldsymbol{\nu}}),
\end{align*}
where $\boldsymbol{\Sigma}_{\boldsymbol{\nu}} = \boldsymbol{\Sigma}_{\boldsymbol{\rho}_1}  +\boldsymbol{\Sigma}_{\boldsymbol{\rho}_2} + i(\boldsymbol{\Sigma}_{\boldsymbol{\rho}_{12}}^\top - \boldsymbol{\Sigma}_{\boldsymbol{\rho}_{12}})$ and $ \textbf{P}_{\boldsymbol{\nu}} = \boldsymbol{\Sigma}_{\boldsymbol{\rho}_1} - \boldsymbol{\Sigma}_{\boldsymbol{\rho}_2} + i(\boldsymbol{\Sigma}_{\boldsymbol{\rho}_{12}}^\top + \boldsymbol{\Sigma}_{\boldsymbol{\rho}_{12}})$. By denoting $\boldsymbol{\nu}^\top = \begin{pmatrix} \nu_{1,1} & \nu_{2,1} \cdots \nu_{d,d} \end{pmatrix}$, we get that $\boldsymbol{\Sigma}_{\boldsymbol{\nu}}$ has entries of the form,
\begin{align*}
\mathbb{E}\left[ \nu_{j,k} \nu_{l,m}^* \right] &= \textnormal{S}_0\left[ {h}_{j,k}(\textbf{y}_1), {h}_{l,m}(\textbf{y}_1 )\right] +  \textnormal{S}_0\left[ \tilde{{h}}_{j,k}(\textbf{y}_1), \tilde{h}_{l,m}(\textbf{y}_1 )\right]\\
& + \sum_{\tau=1}^\infty \left(  \textnormal{R}_\tau\left[ {h}_{j,k}(\textbf{y}_1), {h}_{l,m}(\textbf{y}_1 )\right] + \textnormal{R}_\tau\left[ \tilde{h}_{j,k}(\textbf{y}_1), \tilde{h}_{l,m}(\textbf{y}_1 )\right]  \right)\\
& + i\left(\textnormal{S}_0\left[ \tilde{{h}}_{j,k}(\textbf{y}_1), {h}_{l,m}(\textbf{y}_1 )\right]  -\textnormal{S}_0\left[ {h}_{j,k}(\textbf{y}_1), \tilde{h}_{l,m}(\textbf{y}_1 )\right]     \right)\\
&+i \sum_{\tau=1}^\infty \left(  \textnormal{R}_\tau\left[ \tilde{h}_{j,k}(\textbf{y}_1), {h}_{l,m}(\textbf{y}_1 )\right] - \textnormal{R}_\tau\left[ {h}_{j,k}(\textbf{y}_1), \tilde{h}_{l,m}(\textbf{y}_1 )\right]  \right),
\end{align*}
where $\textnormal{S}_0$ and $\textnormal{R}_\tau$ are defined in Section \ref{sec:ctimeseries}. The elements of the relation matrix $\textbf{P}_{\boldsymbol{\nu}}$ have the form,
\begin{align*}
\mathbb{E}\left[ \nu_{j,k} \nu_{l,m} \right] &= \textnormal{S}_0\left[ {h}_{j,k}(\textbf{y}_1), {h}_{l,m}(\textbf{y}_1 )\right] -  \textnormal{S}_0\left[ \tilde{{h}}_{j,k}(\textbf{y}_1), \tilde{h}_{l,m}(\textbf{y}_1 )\right]\\
& + \sum_{\tau=1}^\infty \left(  \textnormal{R}_\tau\left[ {h}_{j,k}(\textbf{y}_1), {h}_{l,m}(\textbf{y}_1 )\right] - \textnormal{R}_\tau\left[ \tilde{h}_{j,k}(\textbf{y}_1), \tilde{h}_{l,m}(\textbf{y}_1 )\right]  \right)\\
& + i\left(\textnormal{S}_0\left[ \tilde{{h}}_{j,k}(\textbf{y}_1), {h}_{l,m}(\textbf{y}_1 )\right]  +\textnormal{S}_0\left[ {h}_{j,k}(\textbf{y}_1), \tilde{h}_{l,m}(\textbf{y}_1 )\right]     \right)\\
&+i \sum_{\tau=1}^\infty \left(  \textnormal{R}_\tau\left[ \tilde{h}_{j,k}(\textbf{y}_1), {h}_{l,m}(\textbf{y}_1 )\right] + \textnormal{R}_\tau\left[ {h}_{j,k}(\textbf{y}_1), \tilde{h}_{l,m}(\textbf{y}_1 )\right]  \right).
\end{align*}
Note that the diagonal elements of $\boldsymbol{\Sigma}_{\boldsymbol{\nu}}$ are real-valued and recall that $\tilde{h}_{j,j} =0$ for every $j\in\{1,\ldots d\}$.
\end{proof}

\begin{proof}[Proof of Proposition \ref{pro:cross-disappear}]
With $\gamma$ given by Equation \eqref{eq:rate-gamma}, Condition \eqref{eq:maximum} translates into
\begin{equation}
\label{eq:gamma-bound}
\gamma < \min_{j,k\in I}\left\{ q_{1,k}(1-H_k)+q_{1,j}(1-H_j), 1/2\right\}.
\end{equation}

Using independence of the processes $\boldsymbol{\eta}^{(k)}_{\textnormal{\tiny{\textbullet}}}$ and $\boldsymbol{\eta}^{(j)}_{\textnormal{\tiny{\textbullet}}}$ together with straightforward computations, we obtain
\begin{equation*}
\begin{split}
&\mathbb{E}\left[T^{\gamma-1}\sum_{t=1}^T \left[ \left(\tilde{f}_k(\boldsymbol{\eta}^{(k)}_t)-\mathbb{E}\left(\tilde{f}_k(\boldsymbol{\eta}^{(k)}_t)\right)\right) \left(\tilde{f}_j(\boldsymbol{\eta}^{(j)}_t)-\mathbb{E}\left(\tilde{f}_j(\boldsymbol{\eta}^{(j)}_t)\right)\right)\right]\right]^2 \\
& = T^{2\gamma-2}\sum_{t,s=1}^T \textnormal{S}_0\left[\tilde{f}_k(\boldsymbol{\eta}^{(k)}_t),\tilde{f}_k(\boldsymbol{\eta}^{(k)}_s)\right]\textnormal{S}_0\left[\tilde{f}_j(\boldsymbol{\eta}^{(j)}_t),\tilde{f}_j(\boldsymbol{\eta}^{(j)}_s)\right] \\
&\leq c_1T^{2\gamma-1}\sum_{t=1}^T \left|\textnormal{S}_0\left[\tilde{f}_k(\boldsymbol{\eta}^{(k)}_t),\tilde{f}_k(\boldsymbol{\eta}^{(k)}_1)\right]\textnormal{S}_0\left[\tilde{f}_j(\boldsymbol{\eta}^{(j)}_t),\tilde{f}_j(\boldsymbol{\eta}^{(j)}_1)\right]\right|,
\end{split}
\end{equation*}
where $c_1 \in (0,\infty)$ and the inequality follows from change-of-variables and Fubini-Tonelli theorem.  Note that if
$$
\sum_{t=1}^\infty \left|\textnormal{S}_0\left(\tilde{f}_k(\boldsymbol{\eta}^{(k)}_t),\tilde{f}_k(\boldsymbol{\eta}^{(k)}_1)\right)\textnormal{S}_0\left(\tilde{f}_j(\boldsymbol{\eta}^{(j)}_t),\tilde{f}_j(\boldsymbol{\eta}^{(j)}_1)\right)\right| < +\infty,
$$
the claim follows, since by Equation \eqref{eq:gamma-bound} we have that $\gamma < 1/2$. In particular, this is the case if $f_k(\boldsymbol{\eta}^{(k)})$ or $f_j(\boldsymbol{\eta}^{(j)})$ is short-range dependent. Hence we may assume that $j,k\in I$ and that
$
q_{1,k}(2H_k-2) + q_{1,j}(2H_j-2)\geq -1.
$
Let first
$
q_{1,k}(2H_k-2) + q_{1,j}(2H_j-2)>-1.
$
In this case we get
\begin{equation*}
\begin{split}
& T^{2\gamma-1}\sum_{t=1}^T \left|\textnormal{S}_0\left(\tilde{f}_k(\boldsymbol{\eta}^{(k)}_t),\tilde{f}_k(\boldsymbol{\eta}^{(k)}_1)\right)\textnormal{S}_0\left(\tilde{f}_j(\boldsymbol{\eta}^{(j)}_t),\tilde{f}_j(\boldsymbol{\eta}^{(j)}_1)\right)\right|\\
&\leq c_2T^{2\gamma-1}\sum_{t=1}^T t^{q_{1,k}(2H_k-2) + q_{1,j}(2H_j-2)}\\
&\leq c_2T^{2\gamma +q_{1,k}(2H_k-2) + q_{1,j}(2H_j-2)}
\end{split}
\end{equation*}
which converges to zero by Equation \eqref{eq:gamma-bound}  and $c_2 \in (0,\infty)$.  Similarly, if
$
q_{1,k}(2H_k-2) + q_{1,j}(2H_j-2)= -1
$
we obtain
\begin{equation*}
\begin{split}
& T^{2\gamma-1}\sum_{t=1}^T \left|\textnormal{S}_0\left(\tilde{f}_k(\boldsymbol{\eta}^{(k)}_t),\tilde{f}_k(\boldsymbol{\eta}^{(k)}_1)\right)\textnormal{S}_0\left(\tilde{f}_j(\boldsymbol{\eta}^{(j)}_t),\tilde{f}_j(\boldsymbol{\eta}^{(j)}_1)\right)\right|\\
&\leq c_3 T^{2\gamma-1}\sum_{t=1}^T t^{-1}\leq c_4 T^{2\gamma -1}\log T,
\end{split}
\end{equation*}
which converges to zero since $\gamma < \frac12$ and $c_3, c_4 \in (0, \infty)$.
This verifies Equation \eqref{eq:cross-disappear1}. Treating Equation \eqref{eq:cross-disappear2} similarly concludes the proof.
\end{proof}
\begin{proof}[Proof of Lemma \ref{lma:mu-conv}]
Recall that the real part of the $k$:th component $\tilde{\boldsymbol{\mu}}^{(k)}$ is given by,
$$
\frac{1}{T}\sum_{k=1}^T \left[ \tilde{f}_k(\boldsymbol{\eta}^{(k)}_t)-\mathbb{E}\left(\tilde{f}_k(\boldsymbol{\eta}^{(k)}_t)\right) \right].
$$
As in the proof of Proposition \ref{pro:cross-disappear}, we may assume $k\in I$ and that $q_{1,k}(2H_k-2)>-1$, in which case we have,
\begin{equation*}
\begin{split}
\mathbb{E}\left[ \left(\mathfrak{Re}(\tilde{\boldsymbol{\mu}}^{(k)})\right)^2\right] &\leq \frac{c_1}{T} \sum_{t=1}^T \left|S_0\left[\tilde{f}_k(\boldsymbol{\eta}^{(k)}_t),\tilde{f}_k(\boldsymbol{\eta}^{(k)}_1)\right]\right|\\
&\leq \frac{c_2}{T}\sum_{t=1}^T t^{q_{1,k}(2H_k-2)}
\leq c_3T^{q_{1,k}(2H_k-2)},
\end{split}
\end{equation*}
where $c_1$, $c_2, c_3 \in (0,\infty)$. Similarly, for $j\in I$ and $c_4 \in (0,\infty)$,  we  obtain
$$
\mathbb{E}\left[ \left(\mathfrak{Re}(\tilde{\boldsymbol{\mu}}^{(j)})\right)^2\right]\leq c_4 T^{q_{1,j}(2H_j-2)}.
$$
Hereby, Cauchy-Schwartz inequality yields
\begin{align*}
\mathbb{E}\left[\left|T^{\gamma}\mathfrak{Re}(\tilde{\boldsymbol{\mu}}^{(k)}) \mathfrak{Re}(\tilde{\boldsymbol{\mu}}^{(j)})\right| \right] &\leq  T^{\gamma}\sqrt{\mathbb{E}\left[ \left(\mathfrak{Re}(\tilde{\boldsymbol{\mu}}^{(k)})\right)^2\right]\mathbb{E}\left[ \left(\mathfrak{Re}(\tilde{\boldsymbol{\mu}}^{(j)})\right)^2\right]}\\
&\leq  c_5 T^{\gamma}T^{q_{1,k}(H_k-1) + q_{1,j}(H_j-1)},
\end{align*}
which converges to zero by Equation \eqref{eq:gamma-bound} and $c_5 \in (0,\infty)$. Note that identical arguments remain valid for the imaginary parts. Finally, the short-range dependent cases and the boundary cases $q_{1,k}(2H_k-2)=-1$ can be treated similarly as in the proof of Proposition \ref{pro:cross-disappear}.
\end{proof}

\begin{proof}[Proof of Theorem \ref{thm:limit-longrange}]
Similarly as in the proof of Theorem \ref{thm:covconvergence}, we set
\begin{align*}
 \textbf{y}^\top_\textnormal{\tiny{\textbullet}} = \begin{pmatrix}\boldsymbol{\eta}^{(1)}_{\textnormal{\tiny{\textbullet}}} & \boldsymbol{\eta}^{(2)}_{\textnormal{\tiny{\textbullet}}} & \ldots & \boldsymbol{\eta}^{(2d)}_{\textnormal{\tiny{\textbullet}}} & \boldsymbol{\eta}^{(1)}_{\textnormal{\tiny{\textbullet}}+\tau} & \boldsymbol{\eta}^{(2)}_{\textnormal{\tiny{\textbullet}}+\tau} &  \ldots  & \boldsymbol{\eta}^{(2d)}_{\textnormal{\tiny{\textbullet}}+\tau}  \end{pmatrix}.\end{align*} We also introduce functions $g_j(\textbf{y}_\textnormal{\tiny{\textbullet}}):\mathbb{R}^{4d}\rightarrow \mathbb{R},  j\in\{1,2,\ldots,4d\},$ through connections $g_j(\textbf{y}_\textnormal{\tiny{\textbullet}}) = \tilde{f}_j(\boldsymbol{\eta}^{(j)}_{\textnormal{\tiny{\textbullet}}}) ,j\in\{1,2,\ldots,2d\}$, and $g_{j+2d}(\textbf{y}_\textnormal{\tiny{\textbullet}}) = \tilde{f}_{j}(\boldsymbol{\eta}^{(j)}_{\textnormal{\tiny{\textbullet}}+\tau}),$ $ j\in\{1,2,\ldots,2d\}$. 
Moreover, we denote $f_k(\textbf{y}_t) \coloneqq g_k(\textbf{y}_t) -   \mathbb{E}[g_k(\textbf{y}_t)]$, which gives that $\mathbb{E}\left[f_k(\textbf{y}_t)\right] = 0$ for all $k\in\{1,2,\ldots, 4d\}$.

By following the proof of Theorem \ref{thm:covconvergence}, we can express the symmetrized autocovariance matrix estimator as
\begin{align*}
\hat{\textbf{S}}_\tau[\textbf{X}] = \frac{1}{2(T-\tau)} \sum_{t=1}^{T-\tau}\left( \tilde{\textbf{X}}_t\tilde{\textbf{X}}^\textnormal{H}_{t+\tau} + \tilde{\textbf{X}}_{t+\tau}\tilde{\textbf{X}}^\textnormal{H}_{t} \right) + \tilde{\boldsymbol{\mu}}\tilde{\boldsymbol{\mu}}^{\textnormal{H}}+ \mathcal{O}_p(1/T).
\end{align*}
We begin by showing that the off-diagonal terms vanish. Following the proof of Theorem \ref{thm:covconvergence}, we obtain
\begin{align*}
T^{\gamma}\left(\boldsymbol{\Lambda}_\tau \hat{\textbf{S}}_0[\textbf{X}] - \hat{\textbf{S}}_\tau[\textbf{X}] \right),
\end{align*}
which is equal to,
\begin{align*}
T^{\gamma-1}\sum_{t=1}^{T}& \frac{1}{2}\left( 2\boldsymbol{\Lambda}_\tau  \tilde{\textbf{X}}_t\tilde{\textbf{X}}^\textnormal{H}_{t} - \tilde{\textbf{X}}_t\tilde{\textbf{X}}^\textnormal{H}_{t+\tau} - \tilde{\textbf{X}}_{t+\tau}\tilde{\textbf{X}}^\textnormal{H}_{t} \right) + T^{\gamma}\left(\boldsymbol{\Lambda}_\tau - \textbf{I}_d\right)\tilde{\boldsymbol{\mu}}\tilde{\boldsymbol{\mu}}^{\textnormal{H}} + o_p(1),
\end{align*}
with the convention $\tilde{\textbf{X}}_{k+\tau} = \textbf{0}$, when $k+\tau > T $. By Lemma \ref{lma:mu-conv}, the off-diagonal elements of $T^{\gamma}\tilde{\boldsymbol{\mu}}\tilde{\boldsymbol{\mu}}^{\textnormal{H}}$ vanish. Thus, the off-diagonal element $(j,k)$, $j\neq k$, of $T^{\gamma}(  \hat{\textbf{J}}\hat{\boldsymbol{\Gamma}} - \textbf{I}_d)$ is asymptotically equivalent with $\textbf{H}_{jk} + o_p(1)$,where, independence and the proof of Theorem \ref{thm:covconvergence}, give that,
\begin{align*}
\textbf{H}_{jk} = \frac{T^{\gamma-1}}{\lambda_\tau^{(k)}-\lambda_\tau^{(j)}} \sum_{t=1}^T \left( {h}_{j,k}(\textbf{y}_t) +  i\tilde{h}_{j,k}(\textbf{y}_t)  \right),
\end{align*}
with
\begin{align*}
{h}_{j,k}(\textbf{y}_t) =& \lambda_\tau^{(j)}\left[ f_j(\textbf{y}_t)f_k(\textbf{y}_t) + f_{j+d}(\textbf{y}_t)f_{k+d}(\textbf{y}_t)   \right]
- \frac{1}{2} \left[ f_{j}(\textbf{y}_t)f_{k+2d}(\textbf{y}_t) \right. \\
& \left. +f_{j+d}(\textbf{y}_t) f_{k+3d}(\textbf{y}_t) + f_{j+2d}(\textbf{y}_t) f_{k}(\textbf{y}_t) + f_{j+3d}(\textbf{y}_t) f_{k+d}(\textbf{y}_t)\right],
\end{align*}
and,
\begin{align*}
\tilde{h}_{j,k}(\textbf{y}_t) =& \lambda_\tau^{(j)}\left[ f_{j+d}(\textbf{y}_t)f_k(\textbf{y}_t) - f_{j}(\textbf{y}_t)f_{k+d}(\textbf{y}_t)   \right]
-\frac{1}{2}\left[ f_{j+d}(\textbf{y}_t)f_{k+2d}(\textbf{y}_t) \right.\\
& \left. + f_{j+3d}(\textbf{y}_t) f_{k}(\textbf{y}_t) - f_{j}(\textbf{y}_t) f_{k+3d}(\textbf{y}_t) - f_{j+2d}(\textbf{y}_t) f_{k+d}(\textbf{y}_t)\right].
\end{align*}
Thus, by Proposition \ref{pro:cross-disappear}, $\textbf{H}_{jk}$ converges to zero in probability. Hence it remains to prove the convergence of the diagonal elements. We can write,
\begin{align*}
T^{\gamma}\left(\hat{\textbf{S}}_0[\textbf{X}] - \textbf{I}_d\right) &=  T^{\gamma-1}\sum_{t=1}^{T}\left( \tilde{\textbf{X}}_t\tilde{\textbf{X}}^\textnormal{H}_{t} -  \mathbb{E}\left[ \tilde{\textbf{X}}_t\tilde{\textbf{X}}^\textnormal{H}_{t} \right] \right) +T^{\gamma}\tilde{\boldsymbol{\mu}}\tilde{\boldsymbol{\mu}}^{\textnormal{H}} + o_p(1).
\end{align*}
Thus, the diagonal element $(j,j)$ of $T^{\gamma}(  \hat{\textbf{J}}\hat{\boldsymbol{\Gamma}} - \textbf{I}_d)$ is asymptotically equivalent with
$\frac{1}{2}\textbf{H}_{jj} + \frac{1}{2}T^{\gamma}\left(\tilde{\boldsymbol{\mu}}\tilde{\boldsymbol{\mu}}^{\textnormal{H}}\right)_{jj}+ o_p(1)$, where
\begin{align*}
\textbf{H}_{jj} = T^{\gamma-1} \sum_{t=1}^T \left( h_{j,j}(\textbf{y}_t) - \mathbb{E}\left[ h_{j,j}(\textbf{y}_t)  \right]\right),
\end{align*}
with $h_{j,j}(\textbf{y}_t) =   \left(f_j(\textbf{y}_t)\right)^2 + \left(f_{j+d}(\textbf{y}_t)\right)^2$,
and
$$
\left[\tilde{\boldsymbol{\mu}}\tilde{\boldsymbol{\mu}}^{\textnormal{H}}\right]_{jj} = \left(\mathfrak{Re}[\tilde{\boldsymbol{\mu}}^{(j)}]\right)^2 +\left( \mathfrak{Im}[\tilde{\boldsymbol{\mu}}^{(j)}]\right)^2.
$$
Recall that $\mathbb{E}[ f_j(\textbf{y}_t)]=0$, 
$
\mathfrak{Re}\left(\tilde{\boldsymbol{\mu}}^{(j)}\right) = \frac{1}{T} \sum_{t=1}^T f_j(\textbf{y}_t)
$
and
$
\mathfrak{Im}\left(\tilde{\boldsymbol{\mu}}^{(j)}\right) = \frac{1}{T} \sum_{t=1}^T f_{j+d}(\textbf{y}_t).
$
Hence, due to independence, it suffices to prove that, for every $j\in\{1,2,\ldots,d\}$, we have
\begin{align*}
&T^{\gamma-1} \! \sum_{t=1}^T \left( h_{j,j}(\textbf{y}_t) - \mathbb{E}\left[ h_{j,j}(\textbf{y}_t)  \right]\right) + T^{\gamma-2}\left(\left[\sum_{t=1}^T f_j(\textbf{y}_t)\right]^2  + \left[ \sum_{t=1}^T f_{j+d}(\textbf{y}_t)\right]^2\right) \\
&\xrightarrow[T\rightarrow \infty]{\mathcal{D}} C_{2,j}Z_{q_{2,j}} + C_{1,j}^2Z^2_{q_{1,j}} + C_{2,j+d}Z_{q_{2,j+d}} + C_{1,j+d}^2Z^2_{q_{1,j+d}}.
\end{align*}
The convergence follows from independence and the continuous mapping theorem, if, $\forall j \in \{1,\ldots, 2d\}$, the following two-dimensional vector converges,
\begin{align}
\label{eq:needed-vector-convergence}
\begin{pmatrix}T^{\gamma-1}\sum_{t=1}^T\left( (f_{j}(\textbf{y}_t))^2 - \mathbb{E}\left[ (f_{j}(\textbf{y}_t))^2  \right]\right)\\[0.3em]
  T^{\frac{\gamma}{2}-1}\sum_{t=1}^T f_j(\textbf{y}_t)\end{pmatrix}\xrightarrow[T\rightarrow \infty]{\mathcal{D}} \begin{pmatrix}C_{2,j}Z_{q_{2,j}} \\ C_{1,j} Z_{q_{1,j}}\end{pmatrix}.
\end{align}
 By using Equation \eqref{eq:long-range-variance} for the asymptotic variance, we observe first that
$$
T^{\gamma-1} \sum_{t=1}^T \left( (f_{j}(\textbf{y}_t))^2 - \mathbb{E}\left[ (f_{j}(\textbf{y}_t))^2  \right]\right)
$$
converges to a non-trivial limit only if $\max_{i\in I}q_{2,i}(2H_i-2) = q_{2,j}(2H_j-2)$. Similarly,
$$
T^{\frac{\gamma}{2}-1}\sum_{t=1}^T f_j(\textbf{y}_t)
$$
converges to a non-trivial limit only if
$\max_{i\in I}q_{2,i}(2H_i-2) = q_{1,j}(4H_j-4)$. Finally, if both of these conditions are satisfied, the convergence in Equation \eqref{eq:needed-vector-convergence} follows from \cite[Theorem 4]{bai-taq2}.
\end{proof}

\bibliographystyle{abbrvnat}
\bibliography{tempuncorcomp}

\end{document}